\documentclass[a4paper,11pt]{article}

\usepackage{amssymb,amsmath,latexsym,amsthm}

\usepackage{xcolor}  
\usepackage{url}
\usepackage{hyperref}
\usepackage{cleveref}
\usepackage{enumerate}
\usepackage[square,sort,comma,numbers]{natbib}
\usepackage{dsfont, esint}
\usepackage{diagbox}
 \usepackage{geometry}
\usepackage{graphicx}
\usepackage{colortbl}
\usepackage{bbm}
\usepackage{mathtools}
\usepackage[title]{appendix}
\usepackage{cancel}
\usepackage[normalem]{ulem}

\DeclarePairedDelimiter\floor{\lfloor}{\rfloor}
\usepackage{endnotes}

\renewcommand{\textbf}[1]{\begingroup\bfseries\mathversion{bold}#1\endgroup}

\usepackage{breqn}

\usepackage{fancyhdr}
\usepackage{pstricks,pst-plot,pst-node,pstricks-add}
\usepackage{auto-pst-pdf}

\newtheorem{thma}{Theorem}

\newtheorem{thm}{Theorem}[section]

\newtheorem{corollary}[thm]{Corollary}
\newtheorem{prop}[thm]{Proposition}

\newtheorem{lemma}[thm]{Lemma}
\newtheorem{claim}[thm]{Claim}

\theoremstyle{definition}
\newtheorem{remark}[thm]{Remark}

\newcommand{\R}{\mathbb R}

\newcommand{\Z}{\mathbb Z}
\newcommand{\N}{\mathbb N}

\newcommand{\Vol}{\mathsf{Vol}}
\numberwithin{equation}{section}

\def\XXint#1#2#3{{\setbox0=\hbox{$#1{#2#3}{\int}$}
    \vcenter{\hbox{$#2#3$}}\kern-.5\wd0}}

\allowdisplaybreaks

\makeatletter
\def\blfootnote{\xdef\@thefnmark{}\@footnotetext}
\newcommand{\subjclass}[2][2020]{%
  \let\@oldtitle\@title%
  \gdef\@title{\@oldtitle\footnotetext{#1 \emph{Mathematics subject classification.} #2}}%
}
\makeatother

\date{date}
\begin{document}

\title{On rectifiability of Delone sets in intermediate regularity}
\subjclass{51F30, 30L05, 28A75}
\author{Irene Inoquio-Renteria and Rodolfo Viera}

\newcommand{\Addresses}{{
  \bigskip
  \footnotesize
\textbf{Irene Inoquio-Renter\'{i}a},\\ \textsc{Universidad de La Serena, Departamento de Matemáticas, Avda. Juan Cisternas 1200, La Serena, Chile.}\par\nopagebreak
\textit{E-mail address: }\texttt{irene.inoquio@userena.cl}\\

\textbf{Rodolfo Viera},\\
\textsc{Universidad Central de Chile, Facultad de Ingeniería y Arquitectura, Avda. Francisco de Aguirre 0405, La Serena, Chile.}\par\nopagebreak
\textit{E-mail addresses:}
\texttt{rvieraq.docente@ucen.cl, r.vieraq@gmail.com}
}}
\date{\today}
\maketitle
\begin{abstract}
   In this work, we deal with Delone sets and their rectifiability under different classes of regularity. By pursuing techniques developed by Rivi\`ere and Ye, and Aliste-Prieto, Coronel and Gambaudo, we give sufficient conditions for a specific Delone set to be equivalent to the standard lattice by bijections having regularity in between bi-Lipschitz and bi-Hölder-homogeneous. From this criterion, we extend a result of McMullen by showing that, for any dimension $d\geq 1$, there exists a threshold of moduli of continuity $\mathcal{M}_d$, including the class of the H\"{o}lder ones, such that for every $\omega\in\mathcal{M}_d$, any two Delone sets in $\R^d$ within a certain class
    cannot be distinguished under bi-$\omega$-equivalence. This class accounts for the decreasing rate of the density deviation of a Delone set, with respect to a limit density. We also extend a result due to Aliste, Coronel, and Gambaudo, which establishes that every linearly repetitive Delone set in $\R^d$ is rectifiable, by extending it to a broader class of repetitive behaviors. Moreover, we show that for the modulus of continuity $\omega(t)=t(\log(1/t))^{1/d}$, every $\omega$-repetitive Delone set in this class is equivalent to the standard lattice by a bi-$\omega$-homogeneous map. Finally, we address a continuous problem related to the previous ones about finding solutions to the prescribed volume form equation in intermediate regularity, thereby extending the results of Rivière and Ye. Some interesting research directions are highlighted.
   

\end{abstract}

\maketitle

\section{Introduction}

A \emph{Delone set} in $\R^d$ is a subset  $X$ that is both discrete and coarsely dense in a uniform way; this means that there are positive constants $\varrho,\vartheta$ such that the distance between any two distinct points in $X$ is at least $\varrho$, and any ball of radius $\vartheta$ contains at least one point of $X$. In the last years, there has been a lot of work in understanding Delone sets from different points of view in mathematics, e.g., Dynamical Systems, Metric Geometry, and Mathematical Physics, specially since its applications in the modeling of solid materials and, particularly, since the outstanding discovery of quasicrystals by Schechtmann and his research team in the eighties \cite{quasi}; see \cite{aper} for a more detailed exposition of these topics. (Physical-)Quasicrystals are point configurations that exhibit long-range order without being periodic, and a major task is to determine how well {\em regular/ordered} they are. Given this, the study of the geometry of Delone sets becomes relevant to address this highly non-trivial problem.  
\medskip

A nice approach to measure the {\em regularity} of a Delone set in $\R^d$ is to understand to what extent it differs from a more regular discrete structure, like a $d$-dimensional lattice in $\R^d$. We say that a Delone set $X$ in $\R^d$ is \emph{rectifiable} if it is bi-Lipschitz equivalent to the lattice  $\Z^d$, which means there exists a bijection $F:X \to \Z^d$ and $L\geq 1$ such that:
\[
(\forall x,y\in X):\qquad\frac{1}{L}\|x-y\|\leq\|F(x)-F(y)\|\leq L\|x-y\|.
\]

The question of whether two Delone sets are bi-Lipschitz equivalent was posed by Gromov (see for instance \cite[Section $3.24_+$]{Gromov}), who specifically asked whether every Delone set in the Euclidean plane $\mathbb{R}^2$ is bi-Lipschitz equivalent to $\mathbb{Z}^2$. Independent counterexamples of this were provided by Burago and Kleiner \cite{BurKl}, as well as McMullen \cite{McMu}, and explicit constructions were provided by Cortez and Navas in \cite{CN}; the second-named author extended this in \cite{RV}, where the lower bi-Lipschitz constant is replaced by a more relaxed condition and it is shown that irregular Delone sets still exist under this notion. McMullen further showed that by relaxing the bi-Lipschitz distortion of the distances to a bi-Hölder-homogeneous equivalence, all Delone sets in $\mathbb{R}^d$ become equivalent. Years later, Burago and Kleiner \cite{BurK2} provided affirmative solutions to Gromov's question and introduced a rectifiability criterion in two dimensions. This result was subsequently extended to higher dimensions by Aliste, Coronel, and Gambaudo \cite{linrep}, who used it to show that linearly repetitive Delone sets are always rectifiable; a well-known example of this is the point set of the vertices of the tiles of a Penrose tiling \cite{aper} and the recently found aperiodic mono-tiling \cite{apmono}.
\medskip

Motivated by the work of Burago \& Kleiner, and McMullen, recently in \cite{irregsep}, Dymond and Kalu\v{z}a addressed the problem of the existence of non-rectifiable Delone sets for regularity classes in between bi-Lipschitz and bi-Hölder homogeneous; 
let us make this more precise. Given a modulus of continuity $\omega:(0,1)\to (0,\infty)$ (see Section \ref{ssec: notations} below for the precise definition), we say that a map $F:X\to Y$ between two Delone sets $X, Y$ in $\R^d$, is {\em $\omega$-homogeneous} if there exists a positive constant $C$ such that for all $R>0$
and for all $x,y\in B_R$:
\[
\|F(x)-F(y)\|\leq CR\omega\left(\frac{\|x-y\|}{R}\right).
\]
The map $F$ is called {\em bi-$\omega$-homogeneous} if it is bijective and $F, F^{-1}$ are $\omega$-homogeneous. A Delone set $X\subset\R^d$ is said to be {\em $\omega$-rectifiable} if there exists a bi-$\omega$-homogeneous map from $X$ to $\Z^d$. In \cite{irregsep} Dymond and Kalu\v{z}a introduced this notion of rectifiability and they proved that for every $d\geq 2$, there exists $\alpha=\alpha(d)\in (0,1)$ and a Delone set $X\subset\R^d$ that is not $\omega_{\alpha}$-rectifiable, for the modulus of continuity
\[
\omega_{\alpha}(t):=t\left(\log\frac{1}{t}\right)^{\alpha}.
\]
Additionally, Dymond and Kalu\v{z}a in \cite{omegdisp} show that $\omega_{\alpha(d)}$-equivalence is, in fact, a weaker type of equivalence than bi-Lipschitz, by providing instances of Delone sets that are $\omega_{\alpha}$-rectifiable but not bi-Lipschitz equivalent to $\Z^d$.
\medskip

By employing techniques introduced by Rivi\'ere and  Ye in \cite{RivYe}, as well as  Aliste, Coronel and Gambaudo \cite{linrep}, we give sufficient conditions for a large class of Delone sets to be $\omega$-rectifiable (see Theorem \ref{thm: sufcondomegdel} below). From this, we obtain our first main result, establishing that there is a threshold of moduli of continuity for which two Delone sets belonging to this class cannot be distinguished by $\omega$-equivalence.

\medskip

 Let $X\subset\R^d$ be a Delone set. For $\rho>0$ and a bounded measurable set $A\subset\R^d$, define the density deviation 
    \begin{equation}
        e_{X,\rho}(A):=\max\left\{\frac{\rho\cdot\Vol(A)}{|A\cap X|},\frac{|A\cap X|}{\rho\cdot\Vol(A)}\right\}.
    \end{equation}
    For each $k\in\N$, define $E_{\rho}(k)$ to be the supremum of the $e_{X,\rho}(C)$, where the supremum is taken over all the $d$-dimensional integer hypercubes $C$ with sidelength $k$.

\medskip

For any given $\rho>0$, we denote by $\mathcal{D}_{\rho,d}$ be the class of Delone sets $X$ in $\R^d$ satisfying that:
\begin{equation}\label{eq: mindensdev}
    E_\rho(2^{i-1})-1\lesssim\frac{1}{1+i}.
\end{equation}
    
From the work \cite{linrep}, it is known that $D_{\rho, d}$ contains linearly repetitive Delone sets having density $\rho$. Later, in the proof of Theorem \ref{thm: reprect}, we shall prove that $D_{\rho,d}$ also contains $r(\log r)^{p}$-repetitive Delone sets in $\R^d$, for $p<1/d$.
\medskip

With these notations, our first main result reads as follows.
\begin{thma}\label{thm: A}
    For $d\geq 1$ consider $\mathcal{M}_{d}$ being the class of the moduli of continuity $\omega:(0,1)\to (0,\infty)$ such that
    \begin{equation}\label{eq: thmA}
        \sum_{i=2}^{\infty}\left(\frac{1}{2^{i-1}\omega(1/2^{i-1})}\right)^{1/d}<\infty,
    \end{equation}
    Then for any $\omega\in\mathcal{M}_{d}$ and any $\rho>0$, every Delone set $X\in\mathcal{D}_{\rho,d}$ is $\omega$-rectifiable.
\end{thma}

Examples of moduli of continuity satisfying \eqref{eq: thmA} include the H\"{o}lder ones, and for fixed $d\geq 1$, the family of moduli $\omega_p(t):=t(\log(1/t))^p$, for $p>d$, that lie in between Lipschitz and H\"{o}lder. Observe that the class $\mathcal{M}_d$ satisfying \eqref{eq: thmA} contains neither the Lipschitz moduli of continuity nor those of the form $\omega_{\alpha}(t)=t(\log(1/t))^{\alpha}$, for $0<\alpha\leq d$. It should be noticed that Theorem \ref{thm: A} extends the result of McMullen \cite[Theorem 5.1]{McMu} to Delone sets in $\mathcal{D_{\rho,d}}$, by narrowing the class of moduli of continuity for which two Delone sets within this class could be distinguished. 
\medskip

 We are also interested in exploring the interaction between the order of a Delone set in terms of the geometric notion of repetitivity, and regularity in the sense of $\omega$-equivalence defined in previous paragraphs. For $r>0$ and $x\in X$, we call the set $X\cap B(x,r)$ a {\em $r$-patch of $X$}. A Delone set $X$ is \emph{repetitive} if for every $r>0$ there exists $R= R(r)>0$ such that every  $R$-patch of $X $ contains (the center of) translated copies of every $r$-patch of $X$; the smallest such $R$ is called the {\em repetitivity function} of $X$ and is denoted by $R_X$. In case of $R_X(r)\asymp r$, the Delone set is said to be {\em linearly repetitive}; in the next, the variables $t$ and $r$ will be used for moduli of continuity and repetitivity functions, respectively.
 \medskip
 



We consider repetitive functions based on moduli of continuity, extending the existing framework to explore the properties of Delone sets further:  
Given a modulus of continuity $\omega:(0,1)\to (0,\infty)$ we say that a Delone set $X\subset\R^d$ is {\em $\omega$-repetitive} if its repetitivity function satisfies
\[
R_X(r)\asymp r^2\omega\left(\frac{1}{r}\right)\footnote{Actually the $\omega$-repetitivity could be defined in terms of an increasing function $\phi(r)\gg r$, by letting $R_X(r)\asymp \phi(r)\omega\left(\frac{r}{\phi(r)}\right)$, but we do not see any potential application for this more general definition.};
\]

in particular, if $\omega(t)\asymp t$, then $\omega$-repetitivity coincides with linear repetitivity. 
The $\omega $-repetitivity extends the linearly repetitive property of $X$, by allowing the search radius 
 for repeated patches to grow according 
 to a modulus $\omega$, permitting the repetitive function $R_X(r)=r^2\omega(1/r)$ to grow faster than linear (whenever $\omega(t)\gg t$) to find copies of patches of smaller scale, such as,  when $\omega(t)\asymp t(\log(1/t))$. Observe that in order for a Delone set $X$ to be $\omega$-repetitive, it is necessary that its repetitive function $R_X$ satisfies that $R_X(r)/r^2$ is concave from a large-enough $r$, and $R_X(r)/r^2\longrightarrow 0$ when $r$ goes to $\infty$.
 
\medskip

In this work, for the first time, we connect repetitive Delone sets with $\omega$-rectifiability, for moduli of continuity $\omega$ being asymptotically greater than Lipschitz. In this way, we intend to initiate covering the gap that \cite[Theorem 2.1]{linrep} and \cite{irregsep} leave naturally open. More precisely, we address the following general question: {\em For which moduli of continuity $\omega$ there holds that every $\omega$-repetitive Delone set in $\R^d$ is $\omega$-rectifiable?} To tackle this, let us consider the moduli of continuity $\omega_p(t)= t(\log \frac{1}{t})^p, p>0$ 
(note that $ t \ll \omega_p(t)$ when $t\to 0$). 
In this case, the concept of $\omega_p$-repetitivity is defined by considering the repetitive function $R_X(r)= r(\log r)^{p}$. With this notation in mind, our second main result reads as follows.

\begin{thma}\label{thm: reprect}
For $0\leq p<1/d$, every $\omega_p$-repetitive Delone set in $\R^d$ is rectifiable.
 \end{thma}

Theorem \ref{thm: reprect} points out that linear repetitivity is not the optimal condition to ensure rectifiability, answering a question posed by Cortez and Navas in \cite{CN}, and generalizing additionally \cite[Theorem 2.1]{linrep}. Note that Theorem \ref{thm: reprect} not only include the moduli of continuity $\omega_p$, but also those satisfying that $\omega\ll \omega_p$ for some $0< p< 1/d$, for instance those of the form $\omega(t)=t\left(\log\log\left(\frac{1}{t}\right)\right)^{\gamma}$, where $\gamma\geq 0$. Thus, Theorem \ref{thm: reprect} is interesting since it establishes a threshold for the class of moduli of continuity where $\omega$-repetitivity implies rectifiability. A similar result was recently obtained in \cite{almostlin} where linear repetitivity is replaced by $\varepsilon$-linear repetitivity, which describes sets in which small-scale patterns are found within larger regions, with a margin of error controlled by a sufficiently small $\varepsilon$.; however, this notion is not directly related to our notion of $\omega$-repetitivity.
\medskip

It is worth mentioning that Theorem \ref{thm: reprect} also can be rewritten in terms of the repetitivity function as follows: {\em for any $0\leq p<1/d$, every $r(\log r)^p$-repetitive Delone set in $\R^d$ (i.e., with repetitive function $R_X(r)=r(\log r)^p$), is rectifiable}; the reason we state Theorem \ref{thm: reprect} as we do, is because of its connection with further results that relate $\omega$-repetitivity and $\omega$-rectifiability, and that we present below.
\medskip

Although in the conclusion of Theorem \ref{thm: reprect} the rectifiability is not achieved in the case $p=\frac{1}{d}$, we still are able to show that $\omega_{1/d}$-repetitive Delone sets satisfying the density deviation assumption \eqref{eq: mindensdev}, are equivalent to $\Z^d$ under bi-$\omega_{\frac{1}{d}}$ homogeneous bijections, as in the spirit of \cite[Theorem 2.1]{linrep}.

\begin{thma}\label{thm omegreprec}
   Every $\omega_{1/d}$-repetitive Delone set in $\mathcal{D}_{\rho,d}$ is $\omega_{1/d}$-rectifiable.
\end{thma}
The proofs of Theorems \ref{thm: reprect} and \ref{thm omegreprec} rely on a Theorem of Lagarias and Pleasants in \cite{repquasi} on the estimation of the density deviation for patches in repetitive Delone sets; this estimation differs for $\omega_p$-repetitive Delone sets in the cases  $ 0 \leq p <1/d$ and  $p =1/d$. In view of Theorem \ref{thm omegreprec}, it is natural to ask the cut-off point for the class of moduli of continuity where the property ``$\omega$-repetitivity implies $\omega$-rectifiability" occurs (without considering the cases when $p>d$ since the $\omega_p$-rectifiability is ensured from Theorem \ref{thm: A}, regardless of any repetitive behavior of the Delone set). We think that every $\omega$-repetitive Delone set in $\R^d$ is $\omega$-rectifiable, but this requires a more detailed study of the estimation of the density deviation for different moduli of continuity, as in \cite{repquasi}, remaining open for analysis, for instance, the moduli $\omega_p$ in the case $1/d<p\leq d$.
\medskip

Lastly, we also address a connected problem corresponding to finding solutions for the {\em prescribed volume for equation}:
\begin{equation}\label{eq: jaceqintro}
    \left\{\begin{array}{rcl}
    \forall E\text{  open  set in }  [0,1]^d,& \displaystyle\int_E f(x)dx=\mathsf{Vol}(\Phi(E)),\\
    \Phi(x)=x,& \text{on }\partial [0,1]^d,
\end{array}\right.
\end{equation}

with a certain gain of regularity, where $f$ is a non-negative integrable function with {\em total mass $1$}\footnote{For us, a non-negative integrable function $f:[0,1]^d\to\R$ with {\em total mass $1$} means that $\int_{[0,1]^d}f(x)dx=1$.}. Given a Delone set $X$ in $\R^d$, it is well-known that bi-Lipschitz solutions for (the non-compact version of) \eqref{eq: jaceqintro}, for a suitable function $f=f_X\in L^{\infty}(\R^d)$ will produce a bi-Lipschitz bijection from $X$ to $\Z^d$; see for instance \cite{linrep, BurK2}, and \cite{sol} in the case of tilings of the plane. 
\medskip

In \cite[Theorem 2]{RivYe}, the authors prove that for any $L^{\infty}$-density function (see \eqref{eq: densfunc} below), the equation \eqref{eq: jaceqintro} has bi-Hölder solutions; moreover, they also establish sufficient conditions for a density function $f$ such that the equation \eqref{eq: jaceqintro} has bi-Lipschitz solutions, by means of suitably controlling the oscillations of $f$.  In this work, we extend this result in the sense that, for any density function $f\in L^{\infty}([0,1]^d)$, we can find a solution for the prescribed volume form equation with more regularity than bi-Hölder.

\begin{thma}\label{thm: riveyeomegaintro}
Let $\omega:(0,1)\to (0,\infty)$ be a modulus of continuity and assume there is an increasing positive function $\varphi:(0,\infty)\to (0,\infty)$ such that
\begin{equation}\label{eq: osc1}
    \int_0^{1/2}\frac{\varphi(t)}{\omega(t)}dt<\infty\quad\text{and}\quad \sum_{k\geq 1}\varphi(2^{-k})<\infty.
\end{equation}
If $f$ is a density function verifying the following control over its oscillations:
\begin{equation}\label{eq: oscD}
    (\forall x\in [0,1]^d,\forall t>0),\quad\fint_{[0,1]^d\cap B(x,t)}\left|f(y)-\fint_{[0,1]^d\cap B(x,t)} f(z) dz
\right|dy\leq\varphi(t),
\end{equation}
then there exists a bi-$\omega$-mapping $\Phi:[0,1]^d\to [0,1]^d$ satisfying \eqref{eq: jaceqintro}.
\end{thma}

The condition \eqref{eq: oscD} has the same flavor that \eqref{eq: thmA} and it is satisfied, for instance, for Hölder moduli of continuity, but also for those of the form $\omega_{1,p}(t)\asymp t(\log (1/t))^p,\ \omega_{2,p}\asymp t\log(1/t)(\log\log (1/t))^p$, and so on, with $p>1$, and every increasing function $\varphi:(0,\infty)\to(0,\infty)$ with $\sum_{k\geq 1}\varphi(2^{-k})<\infty$. It would be interesting to know the exact class of moduli of continuity for which we can always find solutions for \eqref{eq: jaceqintro}; in view of this discussion, a plausible conjecture is that bi-$\omega$-regular solutions for \eqref{eq: jaceqintro} exist for $t\log(1/t)\lesssim\omega(t)$.

\subsection{Notations and conventions}\label{ssec: notations}

Throughout this article, the notation 
$ f(t)\ll g(t)$ means that   $ \frac{f(t)}{g(t)}\to 0, $ when $t\to 0$. The notation
$f\lesssim g$ means that there exists a constant $C>0$ such that  $f\leq C g. $ The notation  $f\asymp g$ means $f\lesssim g\lesssim f.$ Henceforth, by $\log$ we mean $\log_2$.
\medskip

An increasing function $\omega: (0,a)\to (0,\infty)$ is called a {\em modulus of continuity} if it is concave, and $\lim_{t\to 0}\omega(t)=0$. Given two moduli of continuity $\omega_1,\omega_2: (0,1)\to (0,\infty)$, we say that {\em $\omega_2$ is asymptotically larger than $\omega_1$} when $\omega_1(t)\ll\omega_2(t)$, when $t\to 0^+$. 
A modulus of continuity $\omega $  is said to be asymptotically between Lipschitz and H\"older if, when  $t\to 0$, it satisfies  $t\ll\omega (t) \ll t^\alpha, $ for all $\alpha \in (0,1)$. Some examples of moduli of continuity that lie asymptotically between  Lipschitz  and  H\"older  are; for $\gamma>1$: 
$$ t (\log(1/t))^{1/\gamma}\ll t(\log(1/t))\ll t(\log(1/t))(\log \log (1/t))\ll t\log (1/t))^{\gamma}. 
 $$ 

\medskip

Given a modulus of continuity $\omega:(0,a)\to (0,\infty)$, we say that a map $\Phi:[0,1]^d\to\R^d$ is a {\em $\omega$-mapping} (or that it is {\em $\omega$-regular}) if there is some positive constant $C_{\omega}$ such that:
\begin{equation}\label{eq: omegamap}
(\forall x,y\in[0,1]^d\text{ with }\|x-y\|<a),\qquad    \|\Phi(x)-\Phi(y)\|\leq C_{\omega}\cdot\omega(\|x-y\|).
\end{equation}
An homeomorphism $\Phi:[0,1]^d\to\R^d$ is called a {\em bi-$\omega$-mapping} (or bi-$\omega$-regular) if both $\Phi$ and $\Phi^{-1}$ are $\omega$-mappings. Since, for our purposes, the exact value of $a>0$ is irrelevant, we only consider the case $a=1$ from now on.
\medskip

For a {\em density function}, we mean a non-negative integrable function. As customary, we use the notation
 \[
 \fint_A f(x)dx:=\frac{1}{\Vol(A)}\int_A f(x)dx,\qquad\text{where }A\subset\R^d\text{ is a measurable set,}
 \]
where $\Vol$ stands for the Lebesgue measure in $\R^d$. All the squares we consider have sides parallel to the coordinate axes by default. 
\medskip

Given a finite set $F$, we denote its number of points by $|F|$.

\section{A sufficient condition for a Delone set to be $\omega$-rectifiable}\label{sec: rectDel}

Given a modulus of continuity $\omega:(0,1)\to (0,\infty)$, we establish a criterion ensuring the $\omega$-rectifiability of a Delone set provided a strong control over the oscillations of the density of the points in the patches of a Delone set is satisfied. More precisely, if the densities of points in $2^{i}$-patches of a Delone set converge rapidly enough to a limit density, where the density deviation is controlled by a modulus of continuity $\omega$, then the Delone set must be $\omega$-rectifiable. The conditions that we find here, though weaker than the corresponding ones in \cite[Theorem 3.1]{linrep}, still ensure regularity for a Delone set, which actually turns out to be stronger than the minimal expected H\"{o}lder-type regularity obtained in \cite[Theorem 5.1]{McMu}. 

\begin{thm}\label{thm: sufcondomegdel}
    Let $X\subset\R^d$ be a Delone set and $\omega:(0,1)\to\ (0,\infty)$ be a modulus of continuity such that $t\lesssim\omega(t)$. For $\rho>0$, consider the density deviation 
    \begin{equation}
        e_{X,\rho}(A):=\max\left\{\frac{\rho\cdot\Vol(A)}{|A\cap X|},\frac{|A\cap X|}{\rho\cdot\Vol(A)}\right\},\quad\text{for }A\subset\R^d\text{ bounded measurable}.
    \end{equation}
    For each $k\in\N$, define $E_{X,\rho}(k)$ be the supremum of the $e_{X,\rho}(C)$, where the supremum is taken over all the $d$-dimensional integer hypercubes $C$ with sidelength $k$. Assume that $X$ verifies 
    \[
    E_{X,\rho}(2^{i-1})-1\lesssim 1/(1+i).
    \]
    If there exists $\rho>0$ such that 
    \begin{equation}\label{eq: omegadelone}
        \sum_{i=1}^{\infty}(1+i)^{\alpha}\left(\frac{E_{X,\rho}(2^i)-1}{(2^{i-1}\omega(2^{-(i-1)}))^{1/d}}\right)<\infty,
    \end{equation}
    for some $0<\alpha\leq 1$, then $X$ is $\omega$-rectifiable.
\end{thm}


Observe that Theorem \ref{thm: sufcondomegdel} generalizes \cite[Theorem 1.3]{BurK2} by encompassing a broader class of density oscillations for a Delone set. In particular, if $\omega(t)\asymp t$, then \eqref{eq: omegadelone} yields that:
    \[
   \sum_{i\geq 1}\left(E_{X,\rho}(2^i)-1\right)<\infty,
    \]
    implying the sufficient condition that appears in \cite[Theorem 1.3]{BurK2}.
\medskip

As we mentioned in the Introduction, Theorems \ref{thm: A}, \ref{thm: reprect} and \ref{thm omegreprec} are consequences of Theorem \ref{thm: sufcondomegdel}. We start by proving these Theorems and in Section \ref{ssec: prrofthmsuf} below we prove Theorem \ref{thm: sufcondomegdel}. We recall to the reader that
\begin{equation}\label{eq: mindensdevpr}
    X\in\mathcal{D}_{\rho,d}\quad\Leftrightarrow\quad E_{X,\rho}(2^{j-1})-1\lesssim\frac{1}{1+j}.
\end{equation}

\subsection{Proof of Theorem \ref{thm: A}}\label{sec: Theorem A}
Before proving Theorem \ref{thm: A}, we first observe that, although we are requiring the condition $X\in\mathcal{D}_{\rho,d}$, it can be shown that the density deviation $E_{\rho}(2^{j})$ is bounded for any Delone set. 
    
\begin{lemma}\label{claim: boundev}
    Under the notations of Section \ref{sec: rectDel}, if $X\subset\R^d$ is a Delone set, then the sequence $(E_{X,\rho}(2^i))_{i\geq 1}$ is bounded.
\end{lemma}

\begin{proof}
    Otherwise, there must exists a sequence of cubes $(B_i)_{i\geq 1}$, where $B_i$ has sidelength $2^i$ for each $i\geq 1$, and such that one of the sequences 
    \[
    a_i:=\frac{|X\cap B_i|}{\rho \Vol(B_i)},\qquad b_i:=\frac{\rho \Vol(B_i)}{|X\cap B_i|}\qquad\text{where }i\geq 1
    \]
   
    is unbounded. If $(a_i)_{i\geq 1}$ is unbounded, then $X$ is not uniformly discrete, and if $(b_i)_{i\geq 1}$ is unbounded, then $X$ cannot be coarsely dense. Therefore $(E_{X,\rho}(2^i))_{i\geq 1}$ must be bounded.
\end{proof}

\begin{proof}[Proof of Theorem \ref{thm: A}]
    Let $X\in\mathcal{D}_{\rho,d}$ and $\omega\in\mathcal{M}_d$ like in Theorem \ref{thm: A}. Then we have that $0\leq E_{X,\rho}(2^{i-1})-1\lesssim \frac{1}{i+1}$, and thus, for any $\alpha\in (0,1)$ be fixed we have:
    \[
    \sum_{i\geq 2}\frac{(i+1)^{\alpha}\left(E_{X,\rho}(2^{i-1})-1\right)}{\left(2^{i-1}\omega\left(\frac{1}{2^{i-1}}\right)\right)^{1/d}}\lesssim\sum_{i\geq 2}\frac{1}{\left(2^{i-1}\omega\left(\frac{1}{2^{i-1}}\right)\right)^{1/d}}<\infty,
    \]
    where in the last step we use that $\frac{1}{(1+i)^{\beta}}\leq 1$, for every fixed $0<\beta\leq 1$ and any $i\geq 2$. Then, the conclusion of Theorem \ref{thm: A} follows from Theorem \ref{thm: sufcondomegdel}.
\end{proof}

\subsection{On repetitiveness and rectifiability: Proof of Theorems \ref{thm: reprect} and \ref{thm omegreprec}}\label{ssec: omegarec}

This Section is devoted to proving  Theorems \ref{thm: reprect} and \ref{thm omegreprec}. The proofs rely on a theorem of Lagarias and Pleasants concerning the estimation of the density deviation for Delone sets in $\R^d$ that are $\omega_p$-repetitive, where $\omega_p(t)=t(\log(1/t))^{p}$; concretely, in our notations, Lagarias-Pleasants' Theorem can be restated as follows.

\begin{prop}[{\em Cf.}  Theorem 5.3 in \cite{repquasi}]\label{prop: replag}
    If $X$ is a $\omega_p$-repetitive Delone set in $\R^d$, then for every $j\geq 1$,
    \begin{equation}\label{eq: replag}
    E_{X,\rho}(2^j)-1\lesssim \prod_{k=2}^{m_j}\left(1-\frac{c_1}{(\log U_k)^{p d}}\right),
    \end{equation}
    where $U_{k+1}:=CU_k(\log(U_k))^p,\ C=4(2\sqrt{d}+1),\ U_1\geq C$, $0<c_1<1$ are universal constants, and $m_j:=\max\{m: U_m\leq 2^j\}$. 
\end{prop}
The following result establishes an asymptotic control for $m_j$; this control is key for estimating the product in \eqref{eq: replag}, and in consequence \eqref{eq: omegadelone}, since:
\[
\prod_{k=2}^{N}\left(1-\frac{c_1}{(\log U_k)^{p d}}\right)\leq\prod_{k=2}^{M}\left(1-\frac{c_1}{(\log U_k)^{p d}}\right)\quad\text{provided }M\leq N.
\]
In the next, recall that $\log$ stands for $\log_2$.

\begin{lemma}\label{cl: omegalag}
    Under the notations of Proposition \ref{prop: replag}, we have that 
    \begin{equation}\label{eq: estmj}
    m_j\asymp \frac{j}{\log j}.
    \end{equation}
\end{lemma}
\begin{proof}
Let $j_0\in\N$ be the smallest positive integer such that $2^{j_0}<U_1<2^{j_0+1}$. By definition of $U_2$, we have that
\[
C2^{j_0}j_0^p<U_2<C2^{j_0+1}(j_0+1)^p\quad\Leftrightarrow\quad 2^{j_0+p\log j_0+\log C}<U_2<2^{j_0+1+p\log(j_0+1)+\log C}.
\]

Similarly, by the definition of $U_3$ we get that
\[
C2^{j_0+p\log j_0+\log C}(j_0+p\log j_0+\log C)^p<U_3<C2^{j_0+1+p\log (j_0+1)+\log C}(j_0+p\log j_0+\log C)^p,
\]
or equivalently,
\[
2^{j_0+p\log j_0+p(\log(j_0+p\log j_0+\log C)+2\log C}<U_3< 2^{j_0+1+p\log(j_0+1)+p\log(j_0+1+p\log(j_0+1)+\log C))+2\log C}.
\]

By continuing this process, there holds that $m_j$ must be the largest positive integer such that 
\begin{equation}\label{eq: estkjmj}
    m_j\log C+M_{m_j,j_0}<j<m_j\log C+M_{{m_j,j_0+1}},
\end{equation}

where $M_{N,k_0},\ k_0\in\{j_0,j_0+1\}$ is defined recursively by
\[
M_{N,k_0}:=\left\{\begin{array}{ccc}
    k_0 & \text{if} & N=0\\
     M_{N-1,k_0}+p\log(M_{N-1,k_0}+(N-1)\log C) & \text{if} & N\geq 1
\end{array}\right.
\]

 In this way, we need to estimate $m_j$ and $M_{m_j,j_0}$. The main asymptotic estimate we obtain is that $M_{N,j_0}\asymp N\log N$, which we proceed to prove.
\medskip

\noindent\fbox{{\bf $M_{N,j_0}\gtrsim N\log N$.}} By definition of $M_{N,j_0}$ we have that
\[
M_{N,j_0}-M_{N-1,j_0}=p\log (M_{N-1,j_0}+(N-1)\log C)>p\log ((N-1)\log C)=p\log (N-1)+p\log C.
\]
By summing and telescoping the above, and from the fact that $\sum_{k=1}^N\log k\geq\int_{1}^N\log xdx$, we obtain that
\begin{equation*}
\begin{split}
M_{N,j_0}-j_0>p\sum_{k=1}^N \log k+pN\log C & >p (N\log N-N+1)+pN\log C\\
& >(N\log N)\left(p-\frac{N+1}{N\log N}+\frac{p}{\log N}\right),
\end{split}
\end{equation*}
where we use $C>2$. Now, since $\left(p-\frac{N+1}{N\log N}+\frac{p}{\log N}\right)\to p$, we get that
\[
M_{N,j_0}\gtrsim N\log N,
\]
as claimed.
\medskip

\noindent{\fbox{$M_{N,j_0}\lesssim N\log (N+1)$}}. Let 
\begin{equation*}
    N_0=\max\{2,j_0\}\quad\text{and}\quad C_1=\max\left\{j_0,2p,\frac{\log C}{2},\frac{M_{N_0,j_0}}{2N_0\log (N_0+1)}\right\}.
\end{equation*}
We shall prove that 
\[
(\forall N\geq N_0),\quad M_{N,j_0}\leq 2C_1N\log (N+1).
\]
Firstly, by definition of $C_1\geq j_0$ we have that: 
\[
M_{N_0,j_0}\leq 2C_1N_0\log (N_0+1.)
\]

For the inductive step, if $M_{N-1,j_0}\leq 2C_1N\log(N+1)$, then by definition of $M_{N,j_0}$ the following holds for $N\geq N_0$:
\begin{equation*}
    \begin{split}
        M_{N,j_0} & = M_{N-1,j_0}+p\log\left(M_{N-1,j_0}+(N-1)\log C\right)\\ 
                  & \leq2C_1N\log (N+1)+p\log(2C_1N\log (N+1)+(N-1)\log C)\\
                  & \leq 2C_1N\log (N+1)+p\log (2N(C_1\log (N+1)+1))\\
                  & =2C_1N\log (N+1)+p\log (N+1)+p\log (2(\log (N+1)+1))\\     
                  & \leq 2C_1N\log (N+1)+4p\log (N+2)\leq 2C_1(N+1)\log (N+2), 
    \end{split}
\end{equation*}

where in the last step we use that $C_1\geq 2p$. Therefore, the claim follows by induction over $N\geq N_0$. 
\medskip

From \eqref{eq: estkjmj} and since $M_{N,j_0}\asymp N\log N$, we get that
\[
m_j\log C+m_j\log m_j\asymp j\quad\Leftrightarrow\quad m_j\asymp \frac{j}{\log m_j}\quad\text{(in particular }j\gg m_j).
\]

By continuing the above asymptotic estimations, we get that 
\begin{equation*}
    \begin{split}
        m_j \asymp \frac{j}{\log\left(\frac{j}{\log m_j}\right)}=\frac{j}{\log j-\log\log m_j}\asymp \frac{j}{\log j},
    \end{split}
\end{equation*}

which finishes the proof of the Lemma.
\end{proof}

Also, we need the following claim to estimate the right-hand side in \eqref{eq: replag}, and its proof follows by induction.

\begin{claim}[{\em Cf.}  page 853 in \cite{repquasi}]\label{cl: omegalag2}
   Under the notations of Proposition \ref{prop: replag} it holds that:
   \[
   \log U_{j}<U_1j\log(U_1j)\log\log (U_1j).
   \]
\end{claim}

Now, we are in a position to prove Theorems \ref{thm: reprect} and \ref{thm omegreprec}. In what follows, since we shall take logarithms and exponentials to obtain estimates for the density deviation $E_{X,\rho}(2^{j-1})-1$, we denote by $C$ any positive upper multiplicative bound that does not depend on $j$.

\begin{proof}[Proof of Theorem \ref{thm: reprect}]
    From Proposition \ref{prop: replag} and Lemma \ref{cl: omegalag2}, the following holds for some positive constant $C$:

    \begin{equation}\label{eq: proof1.2.1}
        E_{X,\rho}(2^{j-1})-1\leq C\prod_{k=2}^{m_j}\left(1-\frac{c_1}{(U_1 k \log (U_1k)\log\log (U_1k))^{pd}}\right).
    \end{equation}

After applying $\log$ in both sides of \eqref{eq: proof1.2.1} and the Taylor expansion of $\log(1-x)$, we have that
\begin{equation*}\label{eq: proof1.2.2}
    \begin{split}
\log(E_{X,\rho}(2^{j-1})-1)& \leq\log C+\sum_{k=2}^{m_j}\log\left(1-\frac{c_1}{(U_1 k \log (U_1k)\log\log (U_1k))^{pd}}\right)\\
                           & \leq \log C-c_1\sum_{k=2}^{m_j}\frac{1}{(U_1 k \log (U_1k)\log\log (U_1k))^{pd}}.
    \end{split}
\end{equation*}

Thus, if $p<\frac{1}{d}$, it follows that
\begin{equation*}
\begin{split}
\sum_{k=2}^{m_j}\frac{1}{(U_1 k \log (U_1k)\log\log (U_1k))^{pd}} & \geq \int_{2}^{m_j}\frac{dx}{(U_1 x \log (U_1x)\log\log (U_1x))^{pd}}\\
              & \geq C\int_{2}^{\log m_j}\frac{e^{(1-pd)y}}{(y\log y)^{pd}}dy\\
              & \geq \frac{Ce^{(1-pd)\log m_j}}{(\log m_j\log\log m_j)^{pd}}\\
              & =\frac{Cm_j^{1-pd}}{(\log m_j\log\log m_j)^{pd}},
\end{split}
\end{equation*}
where the third inequality follows after applying integration by parts. From these calculations and Lemma \ref{cl: omegalag}, we get that

\begin{equation*}
  \begin{split} 
\log(E_{X,\rho}(2^{j-1})-1)\leq C\left(1 -\frac{m_{j}^{1-pd}}{(\log m_j\log\log m_j)^{pd}}\right) & \quad\Rightarrow\quad E_{X,\rho}(2^{j-1})-1\leq Ce^{-\frac{(j/\log j)^{1-pd}}{(\log j\log\log j)^{pd}}}\\
                  & \quad\Rightarrow\quad E_{X,\rho}(2^{j-1})-1\leq C e^{-\frac{j^{1-pd}}{\log j(\log\log j)^{pd}}},
\end{split} 
\end{equation*}

where we used that
\[
-\frac{(j/\log j)^{1-pd}}{(\log m_j\log\log m_j)^{pd}}\lesssim -\frac{(j/\log j)^{1-pd}}{(\log j\log\log j)^{pd}}.
\]
From the comparison criterion  of series (by comparing, for instance, with the sequence $a_j=j^{-2}$) we have that 
\[
\sum_{j\geq 4}e^{-\frac{Cj^{1-pd}}{\log j(\log\log j)^{pd}}}<\infty,
\]
and hence from \cite[Theorem 3.1]{linrep} (or equivalently from Theorem \ref{thm: sufcondomegdel} in the case $\omega(t)\asymp t$), we have that $X$ is rectifiable.
\end{proof}

\begin{proof}[Proof of Theorem \ref{thm omegreprec}]
    Let $\omega(t)=\omega_{p}(t)=t(\log (1/t))^{1/d}$ and $X\in\mathcal{D}_{\rho,d}$ (see \eqref{eq: mindensdevpr}). Then, for $0<\alpha<\frac{1}{d^2}$ we have that:
    \[
    \sum_{j\geq 4}(j+1)^{\alpha}\frac{E_{X,\rho}(2^{j-1})-1}{\left(2^{j-1}\omega_{1/d}\left(\frac{1}{2^{j-1}}\right)\right)^{1/d}}\leq \sum_{j\geq4}\frac{1}{(j+1)^{1-\alpha+\frac{1}{d^2}}}<\infty.
    \]
    Thus, the conclusion follows from Proposition \ref{thm: sufcondomegdel}.
    \end{proof}
    \begin{remark}[On the density deviation hypothesis \eqref{eq: mindensdev}]
    In the case $\omega_{1/d}(t)=t(\log(1/t))^{1/d}$, by arguing as in the proof of Theorem \ref{thm: reprect}, the following estimation holds for $r(\log r)^{1/d}$-repetitive Delone sets in $\R^d$:
        \begin{equation}\label{eq: dendev1d}
        \begin{split}
            E_{X,\rho}(2^{j-1})-1\lesssim\frac{1}{\log\log j}.
        \end{split}
    \end{equation}
    Hence, the requirement $X\in\mathcal{D}_{\rho,d}$ is stronger than the density deviation estimate given in \eqref{eq: dendev1d}. In summary, the proof of Theorem \ref{thm: reprect} tell us that $\omega_p$-repetitive Delone sets in $\R^d$, with $0\leq p<1/d$, automatically belong to $\mathcal{D}_{\rho,d}$, while in the proof of Theorem \ref{thm omegreprec}, we have that $X$ belongs to $\mathcal{D}_{\rho,d}$ is actually an assumption that one should put on the Delone set $X$.
    \medskip

    This points out that a certain decreasing rate for the density deviation is necessary for an $\omega$-repetitive Delone set in $\R^d$ with repetitive behavior slower than $\omega_{1/d}$, or in other words, for $r^2\omega(1/r)\gtrsim r(\log r)^{1/d}$, to be $\omega$-rectifiable. In view of this discussion, we ask whether or not the hypothesis $X\in\mathcal{D}_{\rho,d}$ is sharp for Theorems \ref{thm: A} and \ref{thm omegreprec}.
\end{remark}

\subsection{Proof of Theorem \ref{thm: sufcondomegdel}}\label{ssec: prrofthmsuf}

In this Section, we are devoted to proving Theorem \ref{thm: sufcondomegdel}. As it has been extensively discussed in \cite{BurK2, linrep}, Theorem \ref{thm: sufcondomegdel} is a consequence of a continuous version of it; now we are dedicated to stating this version. Given $f\in L^{\infty}(\R^d)$, one is interested in finding solutions $\Phi:\R^d\to\R^d$ to the equation
\begin{equation}\label{eq: planejac}
    (\forall\text{ open and bounded set } A\subset\R^d),\quad\int_{A}f(x)dx=\Vol(\Phi(A)),
\end{equation}

that corresponds to the non-compact version of \cite[equation (1.3)]{RivYe} (see \eqref{eq: jaceq} below). Let $f:\R^d\to\R$ be a positive function such that $f$ and $1/f$ are bounded. Given $\rho>0$, for any cube in $C\subset\R^d$ let $e_{\rho}(C)$ be the {\em density deviation} of $f$ in $C$ with respect to $\rho$, which is given by:
\begin{equation}\label{eq: oscrho}
    e_{\rho}(C):=\max\left\{\frac{\rho}{\fint_C f(x)dx},\frac{\fint_C f(x)dx}{\rho}\right\}.
\end{equation}

Define the {\em error} of the density deviation of $f$, denoted by $E:\N\to\R$, and given by the formula
\begin{equation}\label{eq: errorrho}
    E_{\rho}(k):=\sup(e_{\rho}(C)), 
\end{equation}

where the supremum is taken over all the cubes in $\R^d$ of the form $\prod_{j=1}^d[i_j,i_j+k]$, with $(i_1,\ldots,i_d)\in\Z^d$.
\medskip

With these notations, the continuous version of Theorem \ref{thm: sufcondomegdel} reads as follows.

\begin{prop}\label{prop: omegaplane}
    Let $\omega:(0,1)\to (0,\infty)$ be a modulus of continuity. Let $f:\R^d\to\R$ be a positive function, being constant on each open unit cube with vertices in $\Z^d$, and let $\rho>0$. Assume that $f$ and $1/f$ are bounded. Also, suppose that $E_{\rho}(2^{i-1})-1\lesssim 1/(i+1)$. If 
    \begin{equation}\label{eq: omegaplaneosc}
        \sum_{i\geq 1}(1+i)^{\alpha}\left(\frac{E_{\rho}(2^{i-1})-1}{\left(2^{i-1}\omega\left(\frac{1}{2^{i-1}}\right)\right)^{1/d}}\right)<\infty,
    \end{equation}
    for some $0<\alpha\leq 1$, then there exists a bi-$\omega$-homogeneous bijection $\Phi:\R^d\to\R^d$ solving the equation \eqref{eq: planejac}.
    \end{prop}

To prove this continuous version of Theorem \ref{thm: sufcondomegdel}, we need the following claim that relates the product of the $(1+a_i)$'s with a sum of the $a_i$'s, in the same spirit as Claim A.5.

\begin{claim}\label{lemma: clA5}
    Let $\alpha\in (0,1)$ and $(a_i)_{i\geq 1}$ be a sequence of positive real numbers such that 
    \[
    (\forall i\geq 1),\qquad a_i\leq\frac{\alpha}{i+1}\footnote{It should be noted that, while we specify a constant $\alpha$ here, in the statement of Proposition \ref{prop: omegaplane} we do not. This is compatible with the condition $E(2^{i-1})-1\lesssim1/(i+1)$, since we have $E(2^{i-1})-1\leq \frac{C}{\alpha}\cdot\frac{\alpha}{i+1}$, where $C$ is the positive constant appearing from the condition $\lesssim$; in this case we should apply Claim \ref{lemma: clA5} to the sequence $a_i:=\alpha\cdot(E(2^{i-1})-1)/C$}.
    \]
    Then for every $m\geq 1$ there holds:
    \begin{equation}\label{eq: prodsum}
        \prod_{i=1}^m (1+a_i)\leq 1+\sum_{i=1}^m (1+i)^{\alpha} a_i. 
    \end{equation}
\end{claim}

\begin{proof}
    The proof follows by induction over $m$; indeed, observe that for $m=1$:
    \[
    1+a_1\leq 1+2^{\alpha}a_1\quad\Leftrightarrow\quad 1\leq 2^{\alpha},
    \]
    which is obviously true. Now, assume that the inequality \eqref{eq: prodsum} holds until some arbitrary $m>1$. To the inductive step, we have that
    \begin{equation*}
        \begin{split}
            \prod_{i=1}^{m+1}(1+a_i) & = \left(\prod_{i=1}^m (1+a_i)\right)(1+a_{m+1})\\
                                     & \leq\left(1+\sum_{i=1}^m (1+i)^{\alpha}a_i\right)(1+a_{m+1})\\
                                     & = \left(1+\sum_{i=1}^m (1+i)^{\alpha}a_i\right)+a_{m+1}\left(1+\sum_{i=1}^m (1+i)^{\alpha}a_i\right),
        \end{split}
    \end{equation*}

    where we used the inductive hypothesis in the first inequality. It remains to prove that 
    \begin{equation}\label{eq: prodsum1}
        a_{m+1}\left(1+\sum_{i=1}^m (1+i)^{\alpha}a_i\right)\leq (m+2)^{\alpha}a_{m+1}.
    \end{equation}

    By the hypothesis over the sequence $(a_i)_{i\geq 1}$ we have that
    \[
    \sum_{i=1}^m(1+i)^{\alpha}a_i\leq \alpha\sum_{i=1}^m\frac{1}{(1+i)^{1-\alpha}}\leq\alpha\int_1^{m+2}\frac{dt}{t^{1-\alpha}} =(m+2)^{\alpha}-1.
    \]
    Therefore, the left-hand side of \eqref{eq: prodsum1} can be bounded from above by:
    \[
    a_{m+1}\left(1+(m+2)^{\alpha}-1\right)\leq a_{m+1}(m+2)^{\alpha}.
    \]
    This proves \eqref{eq: prodsum1} and concludes the Claim.
\end{proof}

\medskip

Now, we  focus on  showing Proposition \ref{prop: omegaplane}. For $\overline{n}\in\Z^d$ and $m\in\N$, consider (with the same notation as \eqref{eq: madic}) the cube
\begin{equation}\label{eq: 2icube}
C_{\overline{n},m}:=\prod_{l=1}^d [n_l,n_l+2^m].
\end{equation}
Also, for $1\leq i\leq m$ and $\overline{k}\in\Lambda_{m,i}:=\Z^d\cap\prod_{l=1}^d [0,2^{m-i})$ consider the following sub-cube of $C_{\overline{n},m}$:
\[
C_{\overline{n},m,i,\overline{k}}:=\prod_{l=1}^d
[n_l+k_l2^i,n_l+(k_l+1)2^{i}];
\]
note that for fixed $1\leq i\leq m$, these cubes form a partition of $C_{\overline{n},m}$. 
\medskip

As in \eqref{eq: rivyecubes}, we consider the following sets:
\begin{equation}\label{eq: rivyecubes2}
    \begin{split}
    A_{\overline{n},m,i,\overline{k}}^p(\epsilon)&:=\prod_{l=1}^{p-1}\left[n_l+k_l2^{i},n_l+(k_l+1)2^{i}\right]\times\left[n_p+k_p2^{i},n_p+\left(k_p+\frac{1}{2}2^{i}\right)\right]\\
        &\times\prod_{l=p+1}^d\left[n_l+\left(k_l+\frac{\epsilon_l}{2}\right)2^{i},n_l+\left(k_l+\frac{\epsilon_l+1}{2}\right)2^{i}\right];\\
        B_{\overline{n},m,i,\overline{k}}^p(\epsilon)&:=\prod_{l=1}^{p-1}\left[n_l+k_l2^{i},n_l+(k_l+1)2^{i}\right]\times\left[n_p+\left(k_p+\frac{1}{2}2^{i}\right),n_p+(k_p+1)2^{i}\right]\\
        &\times\prod_{l=p+1}^d\left[n_l+\left(k_l+\frac{\epsilon_l}{2}\right)2^{i},n_l+\left(k_l+\frac{\epsilon_l+1}{2}\right)2^{i}\right];\\
        D_{\overline{n},m,i,\overline{k}}^p(\epsilon)&:=A_{\overline{n},m,i,\overline{k}}^p(\epsilon)\cup B_{\overline{n},m,i,\overline{k}}^p(\epsilon).
    \end{split}
\end{equation}

Keeping the notations of $\alpha_{\overline{n},m,i,\overline{k}}^p(\epsilon),\beta_{\overline{n},m,i,\overline{k}}^p(\epsilon)$ as  \eqref{eq: localdens} for $A_{\overline{n},m,i,\overline{k}}^p(\epsilon)$ and $B_{\overline{n},m,i,\overline{k}}^p(\epsilon)$, the key Lemma in \cite{linrep} in order to control the oscillations of $f$ is given by the next, where for simplicity, from now on, we write $E_\rho$ as just $E$.
\begin{lemma}\label{lem: lemlinrep}
    Under the hypotheses of Proposition \ref{prop: omegaplane}, for every $\epsilon\in\{0,1\}^d$, $1\leq p\leq d,\ m>0$, $\overline{k}\in\Lambda_{m,i}$, where $1\leq i\leq m$, and $\overline{n}\in\Z^d$, there holds that
    \begin{equation}\label{eq: lemalinrep}
        \frac{1}{2(E_{\rho}(2^{i-1}))^2}\leq\alpha_{\overline{n},m,i,\overline{k}}^p(\epsilon)\leq\frac{E_{\rho}(2^{i-1})^2}{2}\quad\text{and}\quad\frac{1}{2(E_{\rho}(2^{i-1}))^2}\leq\beta_{\overline{n},m,i,\overline{k}}^p(\epsilon)\leq\frac{E_{\rho}(2^{i-1})^2}{2}.
    \end{equation}

    In particular, there exists $0<\eta<1$ such that
    \begin{equation}\label{eq: lemlinrep2}
        \eta\leq\alpha_{\overline{n},m,i,\overline{k}}^p(\epsilon)\quad\text{and}\quad\eta\leq\beta_{\overline{n},m,i,\overline{k}}^p(\epsilon).
    \end{equation}
\end{lemma}

From \cite[Lemma 1]{RivYe} (see Lemma \ref{lem: rivye} below), there exists a bi-Lipschitz homeomorphism $\Phi_{\overline{n},m,i,\overline{k}}^p(\epsilon)$ from $D_{\overline{n},m,i,\overline{k}}^p(\epsilon)$ to itself such that 

\begin{itemize}
    \item[i) ] $\Phi_{\overline{n},m,i,\overline{k}}^p(\epsilon)(x)=x$ if $x\in\partial D_{\overline{n},m,i,\overline{k}}^p(\epsilon)$;
    \item[ii) ] $\mathsf{det}(\nabla \Phi_{\overline{n},m,i,\overline{k}}^p(\epsilon))=2\alpha_{\overline{n},m,i,\overline{k}}^p(\epsilon)$ in $A_{\overline{n},m,i,\overline{k}}^p(\epsilon)$;
    \item[iii) ] $\mathsf{det}(\nabla \Phi_{\overline{n},m,i,\overline{k}}^p(\epsilon))=2\beta_{\overline{n},m,i,\overline{k}}^p(\epsilon)$ in $B_{\overline{n},m,i,\overline{k}}^p(\epsilon)$;
    \item[iv) ] $\|\nabla(\Phi_{\overline{n},m,i,\overline{k}}^p(\epsilon)-\mathsf{Id})\|_{L^{\infty}(D_{\overline{n},m,i,\overline{k}}^p(\epsilon))}\leq\frac{C_{\eta}}{2}(E(2^{i-1})^2-E(2^{i-1})^{-2})$;
    \item[v) ] $\|\nabla(\Phi_{\overline{n},m,i,\overline{k}}^p(\epsilon)^{-1}-\mathsf{Id})\|_{L^{\infty}(D_{\overline{n},m,i,\overline{k}}^{p}(\epsilon))}\leq\frac{C_{\eta}}{2}(E(2^{i-1})^2-E(2^{i-1})^{-2})$;
\end{itemize}
Let $\Phi_{\overline{n},m,i,\overline{k}}^p:C_{\overline{n},m,i,\overline{k}}\to C_{\overline{n},m,i,\overline{k}}$ that coincides with $\Phi_{\overline{n},m,i,\overline{k}}^p(\epsilon)$ in $D_{\overline{n},m,i,\overline{k}}^p(\epsilon)$, and $\Phi_{\overline{n},m,i,\overline{k}}:C_{\overline{n},m}\to C_{\overline{n},m}$ be the bi-Lipschitz homeomorphism given by
\[
\Phi_{\overline{n},m,i,\overline{k}}:=\Phi_{\overline{n},m,i,\overline{k}}^d\circ\Phi_{\overline{n},m,i,\overline{k}}^{d-1}\circ\ldots\circ\Phi_{\overline{n},m,i,\overline{k}}^1.
\]
Observe that, for $\Psi\in\{\Phi_{\overline{n},m,i,\overline{k}}, \Phi_{\overline{n},m,i,\overline{k}}^{-1}\}$, we have that 
\begin{equation}\label{eq: mainbound1}
    \|\nabla\Psi\|_{L^{\infty}(C_{\overline{n},m,i,\overline{k}})}\leq\left(1+\frac{C_{\eta}}{2}\left(E(2^{i-1})^2-E(2^{i-1})^{-2}\right)\right)^d.
\end{equation}

Now let $\Phi_{\overline{n},m,i}:C_{\overline{n},m}\to C_{\overline{n},m}$ be the bi-Lipschitz homeomorphism that coincides with $\Phi_{\overline{n},m,i,\overline{k}}$ in each sub-cube $C_{\overline{n},m,i,\overline{k}}$, and finally consider the bi-Lipschitz homeomorphism $\Phi_{\overline{n},m}:C_{\overline{n},m}\to C_{\overline{n},m}$ given by:

\begin{equation}\label{eq: mainbiliphom}
    \Phi_{\overline{n},m}:=\Phi_{\overline{n},m,m}\circ\Phi_{\overline{n},m,m-1}\circ\ldots\circ\Phi_{\overline{n},m,1};
\end{equation}

observe that $\Phi_{\overline{n},m}$ satisfies the following:
\begin{enumerate}
    \item $\Phi_{\overline{n},m}=\mathsf{Id}$ on $\partial C_{\overline{n},m}$;
    \item  
    $\mathsf{Jac}(\Phi_{\overline{n},m})=f/\fint_{C_{\overline{n},m}}f(x)dx$;
\end{enumerate}

Let $(\Omega^{(m)})_{m\geq 1}$ be a sequence of squares centered at the origin such that, for each $m\geq 1$, the square $\Omega^{(m)}$ is of the form $C_{\overline{n}_1,m}$ for some $\overline{n}_1\in\Z^d$. Finally, consider the homeomorphisms $u_m$ from $\Omega^{(m)}$ to itself as follows: if $x\in\Omega^{(m)}$, then
\begin{equation}\label{eq: homsolplane}
    u_{m}(x)=(\Phi_{\overline{n}_1,m}\circ\ldots\circ\Phi_{\overline{n}_m,1})(x),
\end{equation}
where the $\overline{n}_j$'s satisfy that $C_{\overline{n}_m,1}\subset\ldots\subset C_{\overline{n}_1,m}$.
\medskip

Let us fix two positive integers $m\geq i_0$. After partitioning the cube $\Omega^{(m)}$ in a family of sub-cubes $C_{\overline{n}_1,i_0},\ldots,C_{\overline{n}_{m-i_0},i_0}$ with disjoint interiors and with the same sidelength $2^{i_0}$, we get that
\begin{equation}\label{eq: homsolpart}
    \|\nabla u_{m}\|_{L^{\infty}(C_{\overline{n},m})}=\max\{\|\nabla u_{m}\|_{L^{\infty}(C_{\overline{n}_j,i_0})}:\ 1\leq j\leq m-i_0\}
\end{equation}

Then, by the definition \eqref{eq: homsolplane} we get the following estimation for each $1\leq j\leq m-i_0$:
\begin{equation}\label{eq: mainbound2}
\begin{split}
    \|\nabla u_{m}^{\pm 1}\|_{L^{\infty}(C_{\overline{n}_j,i_0})}&\leq\|\nabla\Phi_{\overline{n}_j,i_0}^{\pm 1}\|_{L^{\infty}(C_{\overline{n_j},i_0})}\prod_{i=i_0}^{m}\left(1+\frac{C_{\eta}}{2}\left(E(2^{i-1})^2-E(2^{i-1})^{-2}\right)\right)^d.
    \end{split}
\end{equation}

It should be notice that $\|\nabla\Phi_{\overline{n}_j,i_0}^{\pm 1}\|_{L^{\infty}(C_{\overline{n}_j,i_0})}$ is a constant that does not depend on $\overline{n}_j\in\Z^d$ since, by the same previous argument to obtain \eqref{eq: mainbound2}, it follows that:
\[
\|\nabla\Phi_{\overline{n}_j,i_0}^{\pm 1}\|_{L^{\infty}(C_{\overline{n}_j,i_0})}\leq \prod_{i=1}^{i_0-1}\left(1+\frac{C_{\eta}}{2}\left(E(2^{i-1})^2-E(2^{i-1})^{-2}\right)\right)^d=:C_1.
\]

Now,  put 
\begin{equation}\label{eq: i0}
 a_i=  \frac{C_{\eta}}{2}\left(E(2^{i-1})-1\right)\lesssim  \frac{1}{1+i}
\end{equation}

then from Claim \ref{lemma: clA5}, \eqref{eq: homsolpart}, \eqref{eq: i0} and \eqref{eq: mainbound2} the following holds
\begin{equation}\label{eq: mainbound3}
    \|\nabla u_{m}^{\pm 1}\|_{L^{\infty}(C_{\overline{n},m})}\leq C_2\left(\sum_{i=i_0}^m (i+1)^{\alpha}\left(E(2^{i-1})-1\right)\right)^d,
\end{equation}

where $C_2=C_1\cdot C_{\eta}/2$. Therefore, we see that Proposition \ref{prop: omegaplane} is a consequence of the following Lemma.

 \begin{lemma}\label{lem: omegalocalplane}
        Under the hypotheses of Proposition \ref{prop: omegaplane}, there exists a positive constant $C$ such that for every $\overline{n}\in\Z^d$ and any $m\geq i_0$, where $i_0\in\N$ verifies \eqref{eq: i0}, the bijection $u_m: C_{\overline{n},m}\to C_{\overline{n},m}$ given in \eqref{eq: homsolplane} satisfies that:
        \begin{equation}\label{eq: omegalocal}
            (\forall x,y\in C_{\overline{n},m}),\quad\|\Psi(x)-\Psi(y)\|\leq C\cdot 2^{m}\omega\left(\frac{\|x-y\|}{2^m}\right),
        \end{equation}
        where $\Psi\in\{u_m,u_m^{-1}\}$.
    \end{lemma}

\begin{proof}
    Let $0<\alpha<1$ be fixed such that \eqref{eq: omegaplaneosc} holds. Write $\Omega^{(m)}=C_{\overline{n},m}$ for some $\overline{n}\in\Z^d$, as above. We only verify \eqref{eq: omegalocal} for $u_{m}$; the argument is the same for $u_m^{-1}$. Let $x\neq y$ in $C_{\overline{n},m}$. Then there exists $\ell\in\Z$, with $-\ell+1\leq m$, such that
        
        \[
        2^{-\ell}\leq\|x-y\|\leq 2^{-\ell+1}\leq\sqrt{d}2^m.
        \]
        
        From \eqref{eq: mainbound3} we get that
                 \begin{equation*}
             \begin{split}
                 \|u_m(x)-u_m(y)\|&\leq\|\nabla u_{m}\|_{L^{\infty}(C_{\overline{n},m})}\|x-y\|\\
                      &\leq C_2 \|x-y\|\left(\sum_{i=i_0}^m(i+1)^{\alpha}\left(E(2^{i-1})-1\right)\right)^d\\
                      &= C_2\|x-y\| 2^m \omega\left(\frac{2^{-\ell}}{2^m}\right)\left(\sum_{i=i_0}^m\frac{(i+1)^{\alpha}\left(E(2^{i-1})-1\right)}{(2^{i-1}\omega(2^{-\ell}/2^{i-1}))^{1/d}}\right)^d\\
                      &\leq 
                       C_2\|x-y\|  2^\ell 2^m \omega\left(\frac{\|x-y\|}{2^m}\right)\left(\sum_{i=i_0}^m\frac{(i+1)^{\alpha}\left(E(2^{i-1})-1\right)}{(2^{i-1}\omega(1/2^{i-1}))^{1/d}}\right)^d\\     &\leq 
                      2 C_2 2^m \omega\left(\frac{\|x-y\|}{2^m}\right)\left(\sum_{i=i_0}^m\frac{(i+1)^{\alpha}\left(E(2^{i-1})-1\right)}{(2^{i-1}\omega(1/2^{i-1}))^{1/d}}\right)^d,    
             \end{split}
         \end{equation*}
                  where in the last three steps we use that $\omega:(0,1)\to (0,\infty)$  is increasing and is concave, that  is, 
 $2^{i-1}\omega\left(\frac{2^{-\ell}}{2^{i-1}}\right)\leq 
  2^m \omega\left(\frac{2^{-\ell}}{2^{m}}\right)$ and
  $ 
   \frac{1}{2^\ell} \omega\left(\frac{1}{2^{i-1}}\right)
  \leq \omega\left(\frac{2^{-\ell}}{2^{i-1}}\right)$
   respectively. Finally, from \eqref{eq: omegaplaneosc} we obtain \eqref{eq: omegalocal} as desired.

\end{proof}

\begin{remark}
    Since $\Phi_{\overline{n},m}$ is $\|\nabla\Phi_{\overline{n},m}\|_{L^{\infty}(C_{\overline{n},m})}$-bi-Lipschitz, then $\Phi_{\overline{n},m}$ is obviously bi-$\omega$-homogeneous for every modulus of continuity $\omega$ being asymptotically greater than Lipschitz; but with this argument the constant defining the homogeneity condition depends on $m$, unlike to \eqref{eq: omegalocal} where the constant $C$ is independent on $m$. 

\end{remark}
    \begin{proof}[Proof of Proposition \ref{prop: omegaplane}]
        Consider $\Psi_m:\R^d\to\R^d$ defined by
        
        \begin{equation*}
  \Psi_m|_{\Omega^{(m)}}=u_m\quad\text{and}\quad\Psi_m|_{\R^d\setminus\Omega^{(m)}}=\mathsf{Id}.
        \end{equation*}
        
        Then, from Lemma \ref{lem: omegalocalplane}, the maps $\Psi_m$ are bi-$\omega$-homogeneous having the same constant $C$ given in \eqref{eq: omegalocal}. In particular, since $f\in L^{\infty}(\R^d)$, the sequence $(\Psi_m)_{m\geq 1}$ is uniformly bounded and, by \eqref{eq: omegalocal}, it is equicontinuous in each $x\in\R^d$; thus, from Arzel\`a-Ascoli theorem, converges (under a subsequence) to a bi-$\omega$-homogeneous mapping $\Psi:\R^d\to\R^d$.
        \medskip

        Finally, since $\Psi_m\longrightarrow\Psi$ in the $L^{\infty}$-norm, and that
        \[
       \fint_{C_m} f(x) dx\longrightarrow\rho,
        \]
        then, for every bounded open set $E\subset\R^d$ we get
        \[
\left(\fint_{C_m} f(x) dx\right)\Vol(\Psi_m(E))\longrightarrow\rho\Vol(\Psi(E)).
        \]
        Therefore, the map $\Phi:=\rho\cdot\Psi:\R^d\to\R^d$ is bi-$\omega$-homogeneous and verifies \eqref{eq: planejac}. This finishes the proof of Proposition \ref{prop: omegaplane}.
    \end{proof}

    \begin{remark}
        It is worth mentioning that for $E\subset\R^d$ being an open bounded set, the sequence of volumes $(\Vol(\Psi_m(E)))_{m\geq 1}$ passes to the limit to $\Vol(\Psi(E))$, but $\mathsf{Jac}(\Psi_m)$ does not necessarily pass to the limit $\mathsf{Jac}(\Psi)$, since $\Psi$ might not be Lipschitz.
    \end{remark}

To finish the proof of Theorem \ref{thm: sufcondomegdel} is a consequence of Hall's marriage Theorem and follows the very same lines as \cite[Theorem 1.3]{BurK2} (see also part 3 of \cite[Theorem 5.1]{McMu}).

\begin{appendices}
\section{Resolutions of the prescribed volume form equation in intermediate regularity}\label{sec: jaceq}

This Appendix is dedicated to finding solutions with a gain of regularity to the prescribed volume form equation on the unit cube (see \eqref{eq: jaceq} below). This problem is known to be connected with the rectifiability of Delone sets since the (im)possibility of solving this equation is related to the (im)possibility of finding bijections from a Delone set in $\R^d$ onto $\Z^d$ with some regularity; see for instance \cite{BurK2, McMu}.
\medskip

Let us consider a density function $f:[0,1]^d\to\R$ verifying the following condition:
\begin{equation}\label{eq: densfunc}
    \inf_{x\in [0,1]^d}f(x)>0\quad\text{and}\quad\int_{[0,1]^d}f(x)dx=1.
\end{equation}

Given $f\in\mathsf{Reg}\subset L^1([0,1]^d)$ where $\mathsf{Reg}$ is some regularity class, we are interested in finding a homeomorphism $\Phi:[0,1]^d\to [0,1]^d$ verifying the prescribed volume form equation: 
\begin{equation}\label{eq: jaceq}
    \left\{\begin{array}{rcl}
    \forall E\text{  open  set in }  [0,1]^d,& \displaystyle\int_E f(x)dx=\mathsf{Vol}(\Phi(E)),\\
    \Phi(x)=x,& \text{on }\partial [0,1]^d;
\end{array}\right.
\end{equation}
particularly we are interested in the cases when the regularity class $\mathsf{Reg}$ is $L^{\infty}([0,1]^d)$. Moreover, for such a choice of $f$, we are seeking a certain gain of regularity for the solution $\Phi$; for instance, in \cite{RivYe}, it is shown that if $\mathsf{Reg}=C([0,1]^d)$, then $\Phi,\Phi^{-1}$ are $\alpha$-Holder for every $\alpha\in (0,1)$, and if $\mathsf{Reg}=L^{\infty}([0,1]^d)$ or $\mathsf{BMO}([0,1]^d)$, then $\Phi,\Phi^{-1}$ belong to $C^{0,\alpha}([0,1]^d)$, for $0<\alpha<\alpha_f<1$, where $\alpha_f$ only depends on $f$ (the exponent $\alpha_f$ for the $ L^{\infty}$-case can differ from the $\mathsf{BMO}$-case).

\medskip

Let $f\in L^{\infty}([0,1]^d)$ be a density function verifying \eqref{eq: densfunc} and $\omega:(0,1)\to (0,\infty)$ be a modulus of continuity. We say that $f$ is {\em $\omega$-realizable} if there exists a homeomorphism $\Phi$ from $[0,1]^d$ to itself that is bi-$\omega$-regular (see \eqref{eq: omegamap}) and verifies \eqref{eq: jaceq}. 
  In \cite[Theorem 1.6]{irregsep} it is shown that if $\omega_{\gamma}(t):=t(\log(1/t))^{\gamma}$, then for a small-enough choice of $\gamma\in (0,1)$, there exists a density function $f\in L^{\infty}([0,1]]^d)$ (even continuous) that is non-$\omega_{\gamma}$-realizable. This and the examples produced in \cite{BurKl} naturally leave open the question of, given a modulus of continuity $\omega$, finding conditions for which a function satisfying \eqref{eq: densfunc} is $\omega$-realizable without being the Jacobian of a bi-Lipschitz homeomorphism. In this work, we address this problem for the first time by finding conditions for realisability (see Proposition \ref{thm: osc} below) and by producing examples of functions with such properties (see Section~\ref{ssec: omeganonlipdens} below).
\medskip


\subsection{A reminder on Rivière-Ye's construction}\label{ssec: Riv-Ye}

The following Lemma, which appears in \cite[Lemma 1]{RivYe} is the cornerstone for constructing solutions of \eqref{eq: jaceq} with a desired regularity; we present here a slightly different version of it whose proof is straightforward and left to the reader.
\begin{lemma}[Rivière-Ye's Lemma, extended]\label{lem: rivye}
    Let $D=[0,1]^d$ and $M\geq 2$ be an even natural number. For each $j\in\{1,\ldots,M\}$, consider the boxes $A_j=[0,1]^{d-1}\times [(j-1)/M,j/M]$. Let $\alpha_1,\ldots,\alpha_M$, be positive numbers such that $\sum_{j=1}^M\alpha_j=1$. Then there exists a bi-Lipschitz homeomorphism $\Phi:D\to D$ such that
    \begin{itemize}
        \item[i) ]$\Phi(x)=x$ for $x\in\partial D$;
        \item[ii) ]for each $j=1,\ldots,M/2$ there holds

        \[\mathsf{det}(\nabla\Phi)=\frac{2\alpha_{2j-1}}{\alpha_{2j-1}+\alpha_{2j}}\text{ in }A_{2j-1},\quad\text{and}\quad
        \mathsf{det}(\nabla\Phi)=\frac{2\alpha_{2j}}{\alpha_{2j-1}+\alpha_{2j}}\text{ in }A_{2j};
        \]
        \item[iii) ]$\|\nabla(\Phi-\mathsf{Id})\|_{L^{\infty}(D)}\leq C_{\eta}\cdot\displaystyle\max_{j=1,\ldots,M/2}|\alpha_{2j-1}-\alpha_{2j}|$;
        \item[iv) ]$\|\nabla(\Phi^{-1}-\mathsf{Id})\|_{L^{\infty}(D)}\leq C_{\eta}\cdot\displaystyle\max_{j=1,\ldots,M/2}|\alpha_{2j-1}-\alpha_{2j}|$,\\

        where $0<\eta<\min\alpha_j<1-\eta$ is given,  and $C_{\eta}>0$ only depends on $\eta$.
    \end{itemize}
\end{lemma}

\begin{proof}
    For each $j\in\{1,\ldots, M/2\}$, let $B_j:=A_{2j-1}\cup A_{2j}$. After rescaling by an affine linear map, from Rivière-Ye's Lemma \cite[Lemma 1]{RivYe} we have that there must exist a bi-Lipschitz homeomorphism $\Phi_j:B_j\to B_j$ such that
    \begin{enumerate}
        \item $\Phi_j(x)=x$ for $x\in\partial B_j$;
        \item $\mathsf{det}(\nabla\Phi_j)=\frac{2\alpha_{2j-1}}{\alpha_{2j-1}+\alpha_{2j}}$ in $A_{2j-1}$;
        \item $\mathsf{det}(\nabla\Phi_j)=\frac{2\alpha_{2j}}{\alpha_{2j-1}+\alpha_{2j}}$ in $A_{2j}$;
        \item $\|\nabla(\Psi_j-\mathsf{Id})\|_{L^{\infty}(B_j)}\leq \widetilde{C_{\eta}}\frac{|\alpha_{2j-1}-\alpha_{2j}|}{\alpha_{2j-1}+\alpha_{2j}}\leq \underbrace{\frac{\widetilde{C_{\eta}}}{2\eta}}_{=:C_{\eta}}\cdot|\alpha_{2j-1}-\alpha_{2j}|$,
    \end{enumerate}
    where $\Psi_j\in\{\Phi_j,\Phi_j^{-1}\}$, $\eta>0$ is as in the statement of Lemma \ref{lem: rivye} and $\widetilde{C_{\eta}}>0$ is the constant depending only on $\eta$ given by \cite[Lemma 1]{RivYe}. Then define $\Phi:D\to D$ by letting
    \[
    \Phi(x):=\Phi_j(x)\quad\text{if }x\in B_j.
    \]
    Observe that $\Phi$ is well-defined since the $B_j$'s have disjoint interiors, their union is $D$, and since $\Phi_j$ and $\Phi_{j+1}$ coincide in $B_j\cap B_{j+1}$, for each $j\in\{1,\ldots,M/2\}$. Moreover, since $\Phi_j|_{\partial B_j}=\mathsf{id}$, then $\Phi$ is also a bi-Lipschitz homeomorphism with bi-Lipschitz constant at most the $\max_{1\leq j\leq M/2}\mathsf{biLip}(\Phi_j)$. 
    \medskip

    Finally, parts iii) and iv) follow from the fact that
    \[
    \|\nabla(\Phi^{\pm}-\mathsf{Id})\|_{L^{\infty}(D)}=\max_{1\leq j\leq M/2}\|\nabla(\Phi_j^{\pm}-\mathsf{Id})\|_{L^{\infty}(B_j)}.
    \]
\end{proof}
To provide a self-contained exposition, we briefly describe Rivière-Ye's method for constructing solutions of \eqref{eq: jaceq}. To avoid overloading the notation and to simplify some computations, we consider the case $M=2$; in Section \ref{ssec: omeganonlipdens} below, we come back to the general case given in Lemma \ref{lem: rivye}. For $\overline{n}=(n_1,\ldots, n_d)\in\Z^d$ and $i\in\N$, consider the dyadic decomposition of the unit cube given by the family:
\begin{equation}\label{eq: madic}
    C_{\overline{n},i}:=\prod_{1\leq l\leq d}\left[\frac{n_l}{2^i},\frac{n_l+1}{2^i}\right].
\end{equation}



Let $f\in L^1([0,1]^d)$ verifying \eqref{eq: densfunc}. For each $i\in\N$, consider $f_i\in L^{\infty}([0,1]^d)$ that is constant in each cube $C_{\overline{n},i}$ and is given by:

\begin{equation}\label{eq: approx}
    f_i(x)=
    \fint_{C_{\overline{n},i} }f(y) dy 
\end{equation}

Observe that by Lebesgue's differentiation theorem, $f_i$ converges to $f$ in the $L^1$-norm.
\medskip

For $\epsilon=(\epsilon_1,\ldots,\epsilon_d)\in\{0,1\}^d$ and $1\leq p\leq d$, we write $A_{\overline{n},i}^p(\epsilon)$, $B_{\overline{n},i}^p(\epsilon)$ and $B_{\overline{n},i}^p(\epsilon)$ for the following subsets of $C_{\overline{n},i}$:
\begin{equation}\label{eq: rivyecubes}
    \begin{split}
        A_{\overline{n},i}^p(\epsilon)&:=\prod_{l=1}^{p-1}\left[\frac{n_l}{2
        ^i},\frac{n_l+1}{2^i}\right]\times\left[\frac{n_p}{2^i},\frac{2n_p+1}{2^{i+1}}\right]\times\prod_{l=p+1}^d\left[\frac{n_l}{2^i}+\frac{\epsilon_l}{2^{i+1}},\frac{n_l}{2^i}+\frac{\epsilon_l+1}{2^{i+1}}\right];\\
        B_{\overline{n},i}^p(\epsilon)&:=\prod_{l=1}^{p-1}\left[\frac{n_l}{2
        ^i},\frac{n_l+1}{2^i}\right]\times\left[\frac{2n_p+1}{2^{i+1}},\frac{n_p}{2^{i}}\right]\times\prod_{l=p+1}^d\left[\frac{n_l}{2^i}+\frac{\epsilon_l}{2^{i+1}},\frac{n_l}{2^i}+\frac{\epsilon_l+1}{2^{i+1}}\right];\\
        D_{\overline{n},i}^p(\epsilon)&:=A_{\overline{n},i}^p(\epsilon)\cup B_{\overline{n},i}^p(\epsilon).
    \end{split}
\end{equation}

Also consider the following positive numbers:
\begin{equation}\label{eq: localdens}
    \alpha_{\overline{n},i}^p(\epsilon):=\frac{\int_{A_{\overline{n},i}^p(\epsilon)}f(x)dx}{\int_{D_{\overline{n},i}^p(\epsilon)}f(x)dx},\quad\text{and}\quad\beta_{\overline{n},i}^p(\epsilon):=\frac{\int_{B_{\overline{n},i}^p(\epsilon)}f(x)dx}{\int_{D_{\overline{n},i}^p(\epsilon)}f(x)dx}.
\end{equation}

As a consequence of Lemma \ref{lem: rivye}, there exist a homeomorphism $\Phi_{\overline{n},i}^p(\epsilon)$ from $D_{\overline{n},i}^p(\epsilon)$ to itself such that:
\begin{itemize}
    \item[i) ] $\Phi_{\overline{n},i}^p(\epsilon)(x)=x$ if $x\in\partial D_{\overline{n},i}^p(\epsilon)$;
    \item[ii) ] $\mathsf{det}(\nabla \Phi_{\overline{n},i}^p(\epsilon))=2\alpha_{\overline{n},i}^p(\epsilon)$ in $A_{\overline{n},i}^p(\epsilon)$;
    \item[iii) ] $\mathsf{det}(\nabla \Phi_{\overline{n},i}^p(\epsilon))=2\beta_{\overline{n},i}^p(\epsilon)$ in $B_{\overline{n},i}^p(\epsilon)$;
    \item[iv) ] $\|\nabla(\Phi_{\overline{n},i}^p(\epsilon)-\mathsf{Id})\|_{L^{\infty}(D_{\overline{n},i})}\leq C_{\eta}\cdot |\alpha_{\overline{n},i}^p(\epsilon)-1/2|$,
\end{itemize}

where $0<\eta<\alpha_{\overline{n},i}^p<1-\eta$, which is possible because the first condition in \eqref{eq: densfunc} and since $f_i\in L^{\infty}(D_{\overline{n},i})$ (a priori, the positive number $\eta$ could vary for differents $D_{\overline{n},i}$'s). Let $\Phi_{\overline{n},i}^p$ be the bi-Lipschitz homeomorphism from $C_{\overline{n},i}$ to itself which coincides with $\Phi_{\overline{n},i}^p(\epsilon)$ in $D_{\overline{n},i}^p(\epsilon)$, and consider the bi-Lipschitz homeomorphism $\Phi_{\overline{n},i}:C_{\overline{n},i}\to C_{\overline{n},i}$ given by:
\[
\Phi_{\overline{n},i}:=\Phi_{\overline{n},i}^d\circ\Phi_{\overline{n},i}^{d-1}\circ\ldots\circ \Phi_{\overline{n},i}^1;
\]

observe that the following holds:
\[
\mathsf{det}(\nabla\Phi_{\overline{n}',i})=2^d\frac{\int_{C_{\overline{n}',i+1}}f(x)dx}{\int_{C_{\overline{n},i}}f(x)dx},\quad\text{for any sub-cube }C_{\overline{n}',i+1}\subset C_{\overline{n},i}.
\]

Finally consider the sequence of homeomorphisms $u_{i}:[0,1]^d\to [0,1]^d$ given by:

\begin{equation}\label{eq: homsol}
    \left\{\begin{array}{lll}
    u_{-1}=\mathsf{Id}\\
u_{i+1}=u_i\circ\Phi_{\overline{n},i},& \text{in }C_{\overline{n},i};
\end{array}\right.
\end{equation}

Observe that for every $i\geq 1$, the homeomorphism $u_i:[0,1]^d\to [0,1]^d$ is bi-Lipschitz and a solution of the PDE:
\begin{equation}\label{eq: approxjaceq}
    \left\{\begin{array}{rcl}
        \mathsf{det}(\nabla u_i)=f_i &\text{in }[0,1]^d\\
         u_i(x)=x &\text{on }\partial [0,1]^d. 
    \end{array}\right.
\end{equation}
In \cite{RivYe}, it is well-argued that good control of the local oscillations of $f$ defined by
\begin{equation}\label{eq: localosc}
   \left|\alpha_{\overline{n},i}^p(\epsilon)-\frac{1}{2}\right|, 
\end{equation}
allow us to find a $L^{\infty}([0,1]^d)$-limit $u:[0,1]^d\to [0,1]^d$ of the sequence $(u_i)_{i\in\N}$ such that, on the one hand, the sequence of prescribed Jacobian equations \eqref{eq: approxjaceq} passes to the limit to the prescribed volume form equation \eqref{eq: jaceq} that is verified for $u$ and, in the other hand, that the limit mapping $u$ possesses a gain of regularity depending on how well-controlled is \eqref{eq: localosc}; we refer to \cite[Section 2.1]{RivYe} for further details.
\subsection{Conditions for the realisability in intermediate regularity}\label{ssec: omegareali}

As a simple adaptation of the techniques developed in \cite{RivYe} we prove that, given a modulus of continuity $\omega$ lying asymptotically between Lipschitz and H\"{o}lder, and a density function $f:[0,1]^d\to (0,\infty)$ satisfying a suitable control (depending on $\omega$) over its local oscillations \eqref{eq: localosc}, then the prescribed volume form equation \eqref{eq: jaceq} admits bi-$\omega$-regular solutions. In what follows, $C(x,t)$ stands for a ($d$-dimensional) cube centered at $x\in\R^d$ and radius $t>0$.

\begin{prop}\label{thm: osc}
Let $\omega:(0,1)\to (0,\infty)$ be a modulus of continuity and assume there is an increasing positive function $\varphi:(0,\infty)\to (0,\infty)$ such that
\begin{equation}\label{eq: osc1}
    \int_0^{1/2}\frac{\varphi(t)}{\omega(t)}dt<\infty\quad\text{and}\quad \sum_{k\geq 1}\varphi(2^{-k})<\infty.
\end{equation}
If $f\in L^1([0,1]^d)$ verifies \eqref{eq: densfunc} and the following control over its oscillations holds:
\begin{equation}\label{eq: osc2}
    (\forall x\in [0,1]^d,\forall t>0),\quad\fint_{[0,1]^d\cap C(x,t)}\left|f(y)-\fint_{[0,1]^d\cap C(x,t)} f(z) dz
\right|dy\leq\varphi(t),
\end{equation}
 then there exists a bi-$\omega$-mapping $\Phi:[0,1]^d\to [0,1]^d$ satisfying \eqref{eq: jaceq}.
\end{prop}

\begin{remark}\label{rmk: osc}
    \begin{enumerate}
        \item It should be noticed that if $t\ll\omega(t)$, as $t \to 0$,  then the condition \eqref{eq: osc1} is weaker than the presented in \cite[Theorem 4]{RivYe}, namely, that:
    \[
    \int_0^1\frac{\varphi(t)}{t}dt<\infty.
    \]
    
     which was used to control the oscillations of the function $f$ in order to get a bi-Lipschitz solution of \eqref{eq: jaceq}. Hence, under weaker conditions on the oscillations of the function $f$, we still find solutions of \eqref{eq: jaceq} but having less regularity than bi-Lipschitz.

        \item Observe that the 1st condition in \eqref{eq: osc1} together with \eqref{eq: osc2} do not imply the continuity of $f$. To obtain the continuity of $f$, we need to use that $\varphi(t)\to 0$.
    \end{enumerate}
\end{remark}



Before proving Proposition \ref{thm: osc}, we need the following inequality for which we include its proof; although it seems to be folklore, we do not have any reference.

\begin{claim}\label{cl: ineq}
Let $(a_k)_{k\geq 1}$ be a sequence of strictly positive numbers that verifies:
\[
\sum_{k=1}^{\infty} a_k<\infty.
\]

Then there exists a positive constant $C$ such that from a large-enough $n\geq 1$: 
\[
\prod_{k=1}^n(1+a_k)\leq C\sum_{k=1}^n a_k.
\]
\end{claim}

\begin{proof}
    Let us write $L:=\sum_{k\geq 1}a_k$. From the inequality $1+x\leq e^{x}$, where $x>0$, it follows that
    \begin{equation}
    \begin{split}
    \prod_{k=1}^n (1+a_k)& \leq e^{\sum_{k=1}^n a_k}\leq e^L=\frac{e^L}{L}\left(\sum_{k=1}^n a_k+\sum_{k=n+1}^{\infty} a_k\right)\leq \left(\frac{2e^L}{L}\right)\sum_{k=1}^n a_k,
    \end{split}
    \end{equation}

where the last inequality holds for a sufficiently large $n\geq 1$. This completes the proof of the claim for $C=2e^L/L$.
\end{proof}

\begin{proof}[Proof of Proposition \ref{thm: osc}]
    Let $\omega: (0,1)\to (0,\infty)$ be a modulus of continuity, $\varphi:(0,\infty)\to (0,\infty)$  be a continuous, increasing positive function satisfying \eqref{eq: osc1}, and $f\in L^{\infty}([0,1]^d)$ verifying \eqref{eq: osc2}; let $\alpha_{\overline{n},i}^p(\epsilon)$ be given in \eqref{eq: localdens}. Then from \eqref{eq: osc2} and from the fact that $A_{\overline{n},i}^p(\epsilon)$ and  $B_{\overline{n},i}^p(\epsilon)$ have the same measure, for every $i\in\N$ we have that:
    \begin{equation}\label{eq: diadosc}
    \begin{split}
    \left|\int_{A_{\overline{n},i}^p(\epsilon)}f(x)dx-\int_{B_{\overline{n},i}^p(\epsilon)}f(x)dx\right|&\leq\int_{A_{\overline{n},i}^p(\epsilon)}\left|f(x)- \fint_{D_{\overline{n},i}^p(\epsilon)} f(y) dy\right|dx\\
    &+\int_{B_{\overline{n},i}^p(\epsilon)}\left|f(x)-\fint_{D_{\overline{n},i}^p(\epsilon)} f(y) dy\right|dx\\
        &\leq 2\int_{D_{\overline{n},i}^p(\epsilon)}\left|f(x)-\fint_{D_{\overline{n},i}^p(\epsilon)}f(y)dy\right|dx\leq 2\varphi(2^{-i})\Vol(D_{\overline{n},i}^p(\epsilon)).
    \end{split}
    \end{equation}
Thus, since $a:=\inf f>0$, we have that

\begin{equation}\label{eq: osccont1}
\begin{split}
    \left|\alpha_{\overline{n},i}^p(\epsilon)-\frac{1}{2}\right|=\frac{1}{2}\left|\alpha_{\overline{n},i}^p(\epsilon)-\beta_{\overline{n},i}^p(\epsilon)\right|=\frac{\left|\int_{A_{\overline{n},i}^p(\epsilon)}f(x)dx-\int_{B_{\overline{n},i}^p(\epsilon)}f(x)dx\right|}{\int_{D_{\overline{n},i}^p(\epsilon)} f(x)dx}\leq \frac{2}{a}\varphi(2^{-i}).
\end{split}    
\end{equation}

Therefore, for $\Phi_{\overline{n},i}: C_{\overline{n},i}\to C_{\overline{n},i}$ and $u_i:[0,1]^d\to [0,1]^d$ given in Section \ref{ssec: Riv-Ye}, we have \eqref{eq: osccont1} implies that:
\begin{equation}\label{eq: osccont0}
    \|\nabla(\Phi_{\overline{n},i}-\mathsf{Id})\|_{L^{\infty}(C_{\overline{n},i})}\leq C_1\varphi(2^{-i}),
\end{equation}

where $C_1:=2C/a$, and $C=C_{\eta}$ is the (uniform) constant given in part iii) of Lemma \ref{lem: rivye} for $0<\eta<\frac{a}{\|f\|_{L^{\infty}([0,1]^d)}}<1-\eta$, where $\|f\|_{L^{\infty}([0,1]^d)}<\infty$ by Remark \ref{rmk: osc}, part 2 (see \eqref{eq: localdens}).
\medskip 

Let $i_{0}\geq 1$ be sufficiently large such that $1+C_1\varphi(2^{-i_0})<2$, and consider $x,y\in [0,1]^d$ with $0<\|x-y\|\leq 2^{-i_0}$.
\medskip
Observe that for $i>i_{0}$, after applying the triangle inequality, the definition of  the $L^{\infty}$-norm, property iii) in Lemma \ref{lem: rivye} and \eqref{eq: osccont0}, we have that \begin{equation}\label{eq: ui0}
        \begin{split}
           \|\nabla u_{i+1}\|_{L^{\infty}(C_{\overline{n},i})}&=\|\nabla(u_{i}\circ\Phi_i)\|_{L^{\infty}(C_{\overline{n},i})}\\
           &\leq\|\nabla u_{i}\|_{L^{\infty}(C_{\overline{n},i-1})}\|(1+C_1\varphi(2^{-i}))\leq\ldots\leq\|\nabla u_{i_{0}}\|_{L^{\infty}(C_{\overline{n},i_0})}\prod_{k=i_0}^{i}(1+C_1\varphi(2^{-k})),
        \end{split}
    \end{equation}
    where the change of indices in the $L^{\infty}$-norms is due to taking the supremum over a sequence of axis-parallel cubes satisfying that $C_{\overline{n},j}\subset C_{\overline{n},j-1}$ for any $j$; thus, after writing 
    \[
    C_{2}:=\|\nabla u_{i_{0}}\|_{\infty}\left(\prod_{k=1}^{i_0-1}(1+C_1\varphi(2^{-k}))\right)^{-1},
    \]
    we get that
    \begin{equation}\label{eq: osccont3}
        \|\nabla u_{i+1}\|_{L^{\infty}(C_{\overline{n},i})}\leq C_{2}\prod_{k=i_0}^{i}(1+C_1\varphi(2^{-k})).
    \end{equation}

    Hence, from \eqref{eq: osccont3}, the triangle inequality, from the fact that $\|\Phi_i-\mathsf{Id}\|_{L^{\infty}(C_{\overline{n},i})}\leq \sqrt{d}/2^i$ and the same considerations about the indices of the $L^{\infty}$-norms mentioned in the previous paragraph, we have that for any $j\geq i\geq i_{0}$ the following holds:
    \begin{equation*}
        \begin{split}
            \|u_{j+1}-u_i\|_{L^{\infty}(C_{\overline{n},i})}&\leq\sum_{l=i}^{j}\|u_{l+1}-u_l\|_{L^{\infty}(C_{\overline{n},l})}\\
                    &=\sum_{l=1}^{j}\|u_l\circ\Phi_l-u_l\|_{L^{\infty}(C_{\overline{n},l})}\\
                    &\leq\sum_{l=i}^{j}\|\nabla u_l\|_{L^{\infty}(C_{\overline{n},l})}\|\Phi_l-\mathsf{Id}\|_{L^{\infty}(C_{\overline{n},l})}\\
                    &\leq\sum_{l=i}^{j}\frac{\|\nabla u_{i_0}\|_{L^{\infty}(C_{\overline{n},i_0})}}{2^{i_0}}\prod_{k=i_0}^{l}\left(\frac{1+C_1\varphi(2^{-k})}{2}\right)\\
                    &=\frac{\|\nabla u_{i_0}\|_{L^{\infty}(C_{\overline{n},i_0})}}{2^{i_0}}\prod_{k=i_0}^{i}\left(\frac{1+C_1\varphi(2^{-k})}{2}\right)\left(1+\sum_{l=i+1}^{j}\left(\prod_{k=i+1}^{l}\left(\frac{1+C_1\varphi(2^{-k})}{2}\right)\right)\right).
        \end{split}
    \end{equation*}

Since $C_1\varphi(2^{-i_0})<1$ and $\varphi$ is increasing, there is a positive constant $C_3$ such that 
\[
1+\sum_{l=i+1}^{j}\left(\prod_{k=i+1}^{l}\left(\frac{1+C_1\varphi(2^{-k})}{2}\right)\right)\leq C_3,
\]
and since $\displaystyle\prod_{k=i_0}^{i}\left(\frac{1+C_1\varphi(2^{-k})}{2}\right)\xrightarrow[i\to\infty]{}0$, the sequence
$(u_j)_{j\ge1}$ is a Cauchy sequence in $C([0,1]^d)$.
Therefore, there exists a function
$ 
u:[0,1]^d\to[0,1]^d
$ 
such that $u_j\to u$ uniformly on $[0,1]^d$, that is, $u$ is the $C([0,1]^d)$-limit of the sequence $(u_j)$.
Letting $j\to\infty$ in the previous estimates, we obtain
\begin{equation}\label{eq: osccont4}
\|u-u_i\|_{L^\infty(C_{\overline n,i})}
\le
C_4\prod_{k=i_0}^{i}
\left(\frac{1+C_1\,\varphi(2^{-k})}{2}\right).
\end{equation}
where 
\[
C_4:=\frac{\|\nabla u_{i_0}\|_{L^{\infty}(C_{\overline{n},i_0})}}{2^{i_0}}\left(\prod_{k=1}^{i_0-1}\frac{1+C_1\varphi(2^{-k})}{2}\right)^{-1}C_3, 
\]

Now, let $i\geq i_{0}$ such that 
    \begin{equation}\label{eq: osccont2}
        2^{-i-1}\leq\|x-y\|\leq 2^{-i}\leq 2^{-i_0}. 
    \end{equation}
    
    Then from \eqref{eq: osccont3}, \eqref{eq: osccont2}, \eqref{eq: osccont4}, \eqref{eq: homsol} and since $\|x-y\|\leq 2^{-i}$, we obtain the following:
\begin{equation}\label{eq: osccont5}
    \begin{split}
        \|u(x)-u(y)\|&\leq\|u(x)-u_i(x)\|+\|u_i(x)-u_i(y)\|+\|u(y)-u_i(y)\|\\
                     &\leq 2C_4\prod_{k=i_0}^{i}\left(\frac{1+C_1\varphi(2^{-k})}{2}\right)+\|\nabla u_{i+1}\|_{L^{\infty}(C_{\overline{n},i})}\|x-y\|\\
                     &\leq 2C_{4}\prod_{k=i_0}^{i}\left(\frac{1+C_1\varphi(2^{-k})}{2}\right)+C_{2}\prod_{k=i_0}^{i}\left(\frac{1+C_1\varphi(2^{-k})}{2}\right)\\
                     &\leq C_5\prod_{k=i_0}^{i}\left(\frac{1+C_1\varphi(2^{-k})}{2}\right)\\
                     &\leq \frac{C_6}{2^i}\sum_{k=i_0}^{i}\varphi(2^{-k}),
    \end{split}
\end{equation}
where $C_5:=\max\{2C_4, C_2\},\ C_6:=C_5C_1$ and in the last step we have used Claim \ref{cl: ineq}  since 
$\sum_{k\geq 1}\varphi(2^{-k})<\infty.$
. To get the $\omega$-regularity of $u$, observe that since $\varphi$ and $\omega$ are increasing, and from the fact that $\omega$ is concave, we obtain the following:
\begin{equation}\label{eq: osccont6}
    \begin{split}
       \sum_{k=i_0}^{i}\varphi(2^{-k})&\leq 2\sum_{k=i_0}^{i}\left(\int_{2^{-k}}^{2^{-k+1}}\frac{\varphi(t)}{\omega(t)}dt\right)2^{k-1}\omega(2^{-k+1})\\
       &\leq 2^{i}\omega(2^{-i+1})\sum_{k=i_0}^{i}\int_{2^{-k}}^{2^{-k+1}}\frac{\varphi(t)}{\omega(t)}dt\\
       &\leq 2\cdot 2^{i}\omega(2^{-i})\left(\int_{0}^{1/2}\frac{\varphi(t)}{\omega(t)}dt\right).
    \end{split}
\end{equation}
Then, by putting together \eqref{eq: osccont5}, \eqref{eq: osccont6} and \eqref{eq: osc1} we conclude that $\|u(x)-u(y)\|\leq C_{\omega}\cdot\omega(\|x-y\|)$, for some positive constant $C_{\omega}$, and therefore $u$ is an $\omega$-mapping.
\medskip
Finally, by noticing that $\|u^{-1}-u_i^{-1}\|_{L^{\infty}(C_{\overline{n},i})}\leq \sqrt{d}2^{-i}$, $\mathsf{det}(\nabla u_i)\geq b>0$ and that from the formula of the inverse of a matrix in terms of its determinant and the co-adjoint matrix, we have:
\[
\|\nabla u_i^{-1}\|_{L^{\infty}(C_{\overline{n},i})}\leq\frac{C}{b}\|\nabla u_i\|_{L^{\infty}(C_{\overline{n},i})},
\]
then, by following the very same preceding lines, we conclude that $u^{-1}$ is a $\omega$-mapping. This finishes the proof of Proposition \ref{thm: osc}.
\end{proof}

In the particular case when $\omega(t)=t\log(1/t)$, if $f$ verifies \eqref{eq: densfunc} and \eqref{eq: osc2} for $\varphi(t)=t^{\alpha}\log(1/t)$, with $\alpha>0$, then there exists a bi-$t\log(1/t)$-mapping solving \eqref{eq: jaceq}. 

\begin{remark}\label{rmk: madic}

From the proof of Proposition \ref{thm: osc}, we observe the following:
\begin{enumerate} 
    \item From \eqref{eq: diadosc}, it is sufficient to verify \eqref{eq: osc2} only for all the cubes of the form $C_{\overline{n},i}$, where $i\in\N$ and $\overline{n}\in\Z^d$. In this case, in order to get the uniform constant $C$ in \eqref{eq: diadosc}, we require the function $f$ also to be (essentially) bounded.
    
        \item The procedure in order to prove Proposition \ref{thm: osc} still works if, instead of considering a dyadic partition of the unit cube, we consider a sequence of uniform decompositions and then apply the general form of Lemma \ref{lem: rivye}. More precisely, for $p\in\N$ being fixed, let $(M_i)_{i=1}^p$ be even positive integers; as a first step, consider a $M_{j_1}$-adic subdivision of the unit cube, as in \eqref{eq: madic}, where $1\leq j_1\leq p$. Then consider a subdivision of each cube in this partition by a $M_{j_2}$-adic decomposition, and so on. In this way, we obtain a sequence of cubes $C_{\overline{n},i}$ as in \eqref{eq: madic}, but with rescaling factor $(M_{j_1}\cdot\ldots\cdot M_{j_i})^{-1}$.
    \medskip

    Thus, the main facts in the discussion in \cite[Section 2.1]{RivYe} still hold, since $\mathsf{diam}(C_{\overline{n},i})=C/(M_{j_1}\cdot\ldots\cdot M_{j_i})$. Therefore, the conclusions of Proposition \ref{thm: osc} also follow if \eqref{eq: osccont1} is satisfied for the cubes of the form $C_{\overline{n},i}$ described in the previous paragraph.

\end{enumerate}    
\end{remark}

\subsection{Examples of realisability in intermediate regularity}\label{ssec: omeganonlipdens}
In this Section, we provide examples of density functions that fail to be bi-Lipschitz realizable, but such that the prescribed volume form equation \eqref{eq: jaceq} admits bi-$\omega$-regular solutions, for $\omega$ lying asymptotically between Lipschitz and H\"{o}lder. The starting point is the construction of Burago and Kleiner in \cite{BurKl} of a density that cannot be the jacobian of a bi-Lipschitz homeomorphism, which we modify slightly in order to control its local oscillations like in \eqref{eq: osc2}. 
\medskip

Firstly, let us briefly describe Burago-Kleiner's construction of a non-bi-Lipschitz realizable density. Given $N\in\N$, consider the rectangle $R_{N}:=[0,1]\times [0,1/N]\subset 
[0,1]^2$; for $i=1,\ldots, N$, denote by $S_i$ the square $\left[\frac{i-1}{N},\frac{i}{N}\right]\times \left[0,\frac{1}{N}\right]$. Given $c>0$, define the density function  $f_{c,N}:R_N\to [1,1+c]$ by
\begin{equation}\label{eq: checkerboard1}
    f_{c,N}(x):=\left\{\begin{array}{rcl}
        1 &\text{if} & x\in S_i\text{ if }i\text{ is an odd number}\\
         1+c &\text{if} & x\in S_i\text{ if }i\text{ is an even number} 
    \end{array}\right.
\end{equation}

In what follows we say that a pair of points $(x,y)\in [0,1]^2$ is {\em $A$-stretched} by a map $\Phi:[0,1]^2\to\R^2$ if 
\[
\|\Phi(x)-\Phi(y)\|\geq A\|x-y\|.
\]

In \cite{BurKl} it is shown that, given $L\geq 1$, there are parameters $k>0, M\in\N$ and $N_0\in\N$ (depending on $L$) such that if $N\geq N_0$ then the following holds: if $\Phi:R_N\to\R^2$ is a $L$-bi-Lipschitz map such that 
\[
\mathsf{Jac}(\Phi)=f_{c,N}\quad\text{a.e.},
\]

then there must exist a pair of points in $R_N$ of the form
\[
\left(\left(\frac{p}{MN},\frac{s}{MN}\right),\left(\frac{q}{MN},\frac{s}{MN}\right)\right),\quad\text{with }p,q\in\Z\cap [0,MN], \text{ and } s\in [0, M]
\]

that is $(1+k)\|\Phi(0,0)-\Phi(1,0)\|$-stretched under $\Phi$. Then, modify $f_{c,N}$ by putting a rescaled copy of it in a thin rectangle at the bottom of each square of the grid $\frac{1}{MN}\Z^2\cap R_N$, and by keeping the function $f_{c,N}$ outside these thin rectangles. In this way, we obtain a new function $f_{c,N}^{(2)}$ taking values in $\{1,1+c\}$ and for which there is a pair of points of the form
\[
\left(\left(\frac{p}{M^2N^2},\frac{s}{M^2N^2}\right),\left(\frac{q}{M^2N^2},\frac{s}{M^2N^2}\right)\right),\quad\text{with }p,q\in\Z\cap [0,M^2N^2], s\in [0, M^2N]
\]
that is $(1+k)^2\|\Phi(0,0)-\Phi(1,0)\|$-stretched under $\Phi$, with the same parameters $M,N,k$ as in the previous step. By continuing this process at smaller and smaller scales, at the $i$-the iteration it is constructed a function $f_{c,N}^{(i)}$ taking values in $\{1,1+c\}$ and for which we find a pair of points in a fine grid of $R_N$ that is $(1+k)^i\|\Phi(0,0)-\Phi(1,0)\|$-stretched by $\Phi$; thus, by choosing  $i_L\in\N$ sufficiently large, we find a pair of points whose stretching factor under an $L$-bi-Lipschitz is more than $L$, and hence such a function $f_{c,N}^{(i_L)}\in L^{\infty}([0,1]^2)$ cannot be realizable as the jacobian of an $L$-bi-Lipschitz map; see \cite[Lemma 3.2]{BurKl} for further details.
\medskip

By taking an increasing sequence of bi-Lipschitz constants $L_n\nearrow\infty$ and a sequence of disjoint squares $S^{(L_n)}\subset [0,1]^2$ converging to a point, from the previous construction we get a sequence of functions $(f_{L_n})_{n\geq 1}$ such that, for each $n\in\N$, the density $f_{L_n}$ cannot be the jacobian of a $L_n$-bi-Lipschitz map. Finally, let $f:[0,1]^2\to \{1,1+c\}$ be any $L^{\infty}$-function satisfying \eqref{eq: densfunc}, and that agrees with $f_{L_n}$ in each square $S^{(L_n)}$; this last density is non-realizable as the jacobian of any bi-Lipschitz map.
\medskip

In view of Proposition \ref{thm: osc}, in order to obtain a density function that is not bi-Lipschitz realizable, but such that \eqref{eq: jaceq} admits bi-$\omega$-regular solutions for a modulus of continuity $\omega:(0,1)\to (0,\infty)$ asymptotically in between Lipschitz and H\"{o}lder, it is sufficient to modify $f_{N,c}$ in such a way that \eqref{eq: osc2} holds for some continuous increasing function $\varphi:(0,\infty)\to (0,\infty)$ verifying \eqref{eq: osc1}. This is done by handling the sizes of the rectangles in the rescaling process of the function $f_{c,N}$; indeed, informally speaking, if at the $i$-th step we have constructed the function $f_{c,N}^{(i)}$ taking values in $\{1,1+c\}$ as in the previous paragraph, then by putting a rescaled copy of $f_{N,c}$ in a small square $S$ determined by the grid $\frac{1}{M^iN^i}\Z$ in which $f_{c,N}^{(i)}$ is defined and attains, say, the value $1+c$, we obtain a new function $f_{N,c}^{(i+1)}$ as described above, and for which the following holds: the value of 

\[
\fint_S f_{c,N}^{(i+1)}(x)dx
\]

becomes closer and closer to $1+c$ provided $h_N^{(i+1)}:=\mathsf{height}(R_{N}^{(i+1)})\searrow 0$, where $R_N^{(i+i)}$ is the thin rectangle at the $(i+1)$-step over which $f_{c,N}$ is rescaled, and $\mathsf{height}(R_{N}^{(i+1)})$ denotes its height. Moreover, assuming that $1/h_N^{(i+1)}\in\N$ is an even number (playing the role of the $M_i$'s in Remark \ref{rmk: madic}), if for each $j=1,\ldots, 1/h_N^{(i+1)}$ we denote $S_{j}^{(i+1)}:=[(j-1)\cdot h_N^{(i+1)},j\cdot h_N^{(i+1)}]\times [0,h_N^{(i+1)}]$, then we get that
\begin{equation}\label{eq: burklosc1}
     \fint_S f_{c,N}^{(i+1)}(x)dx
        =\frac{1}{|S|}\left(|R_{N,1}^{(i+1)}|+(1+c)|R_{N,1+c}^{(i+1)}|\right),
\end{equation}

where $R_{N,1}^{(i+1)}, R_{N,1+c}^{(i+1)}$ are the subsets of $S$ where $f_{c,N}^{(i+1)}$ attains the values $1$ and $1+c$, respectively, and we write $|A|$ for $\Vol(A)$, $A\subset\R^2$. Thus from \eqref{eq: burklosc1}, and taking into in account that 
\[
1\leq\fint_S f_{c,N}^{(i+1)}(x)dx\leq 1+c\quad\text{and}\quad S=R_{N,1}^{(i+1)}\cup R_{N,1+c}^{(i+1)},
\]
we have that
\begin{equation}\label{eq: burklosc2}
\begin{split}
    &\fint_S\left|f_{c,N}^{(i+1)}(y)-\fint_S f_{c,N}^{(i+1)}(x)dx
\right|dy\\
          & = \frac{1}{|S|}\int_{R_{N,1}^{(i+1)}} \left(\fint_S f_{c,N}^{(i+1)}(x)dx-1\right)dy+\frac{1}{|S|}\int_{R_{N,1}^{(i+1)}} \left(1+c-\fint_S f_{c,N}^{(i+1)}(x)dx\right)dy\\
          & =\frac{1}{|S|}\left(\fint_{S} f_{c,N}^{(i+1)}(x)dx\right)\left(|R_{N,1}^{(i+1)}|-|R_{N,1+c}^{(i+1)}|\right)+\frac{1}{|S|}\left((1+c)|R_{N,1+c}^{(i+1)}|-|R_{N,1}^{(i+1)}|\right) \\
          & = \frac{1}{|S|^2}\left(|R_{N,1}^{(i+1)}|^2+c|R_{N,1+c}^{(i+1)}|-(1+c)|R_{N,1+c}^{(i+1)}|^2-\left(|R_{N,1}^{(i+1)}|^2-c|R_{N,1+c}^{(i+1)}|-(1+c)|R_{N,1+c}^{(i+1)}|^2\right)\right)\\
          & = 2c\cdot\frac{|R_{N,1+c}^{(i+1)}|}{|S|}.
\end{split}
\end{equation}

By definition, $R_{N,1+c}^{(i+1)}$ is the disjoint union of  the squares of the form $S_j^{(i+1)}$, where the indices $j$'s are even numbers. Therefore, in view of \eqref{eq: burklosc2} and part 2 of Remark \ref{rmk: madic}, to get \eqref{eq: osc2}, we need to rescale in such a way to obtain
\begin{equation}\label{eq: highomegacheck}
    h_N^{(i+1)}\leq C\cdot l_S\cdot\varphi(l_S),
\end{equation}

where $l_S$ is the side-length of $S$ and $C$ is some positive fixed constant. Observe it is necessary that $\varphi(t)\leq t$ to guarantee that $f_{N,c}^{(i)}$ is not $L$-bi-Lipschitz realizable (where $L>0$ is fixed); for instance, this happens since \eqref{eq: osc2} holds for $\omega(t)=t\log(1/t)$ and $\varphi(t)=t^2$ for $0<t\leq 1$.
\medskip

Now, we proceed to describe our bi-$\omega$-realizable function. Let $L_n\nearrow\infty$ be a sequence of positive numbers (the Lipschitz constants to discard), $\varphi(t)\leq t$ be an increasing continuous function verifying \eqref{eq: osc1}, and $c$ be a positive number. For each $n\in\N$, let $N_n, M_n, k_n$ be the parameters given in Burago-Kleiner's construction, which depend only on $(L_n)_{n\geq 1}$. 
\medskip

For the first iteration of our modification, consider the checkerboard function $f_{N_1,c}^{(1)}:R_{N_1}\to\{1,1+c\}$ like in \eqref{eq: checkerboard1}. At the bottom of each square of the grid $\frac{1}{M_1N_1}\Z^2\cap R_{N_1}$, insert a thin rectangle of length $1/M_1N_1$ and height

\begin{equation}\label{eq: highcheck}
    h_{N_1,\varphi}^{(1)}:=\left\{\begin{array}{lcc}
         \frac{\widetilde{\varphi}(N_1^{-1})}{M_1N_1} & \text{if} & \frac{M_1N_1}{\widetilde{\varphi}(N_1^{-1})}\text{ is even}\\
        \left(\frac{M_1N_1}{\widetilde{\varphi}(N_1^{-1})}+1\right)^{-1} & \text{if }& \frac{M_1N_1}{\widetilde{\varphi}(N_1^{-1})}\text{ is odd}\end{array}\right.
\end{equation}

where $\widetilde{\varphi}(t):=\floor{(\varphi(t))^{-1}}^{-1}$, for $0<t<1$. In each of these rectangles, implant a checkerboard function like in \eqref{eq: checkerboard1} and keep the value of $f_{N_1,c}^{(1)}$ outside of the union of these new thin rectangles; call this new function $f^{(2)}_{N_1,c}$. Since $\varphi(t)\leq t$, we can continue this process inductively, obtaining at the $i$-th step a family of rectangles of length $h_{N_1,c}^{(i-1)}$ and height

\begin{equation}\label{eq: highnew}
h_{N_1,c}^{(i)}:=h_{N_1,c}^{(i-1)}\cdot\widetilde{\varphi}(h_{N_1,c}^{(i-1)})
\end{equation}

defined similarly as in \eqref{eq: highcheck}, and a checkerboard-type function $f_{N_1,c}^{(i)}$ taking values in $\{1,1+c\}$. Then, from the previous discussion, there must exist a sufficiently large positive integer $i_1$ and a checkerboard-type function $f_{N_1,c}^{(i_1)}:[0,1]^2\to\{1,1+c\}$ such that, for every $L_1$-bi-Lipschitz homeomorphism $\Phi:[0,1]^2\to\R^2$ whose jacobian coincides with $f_{N_1,c}^{(i_1)}$ almost everywhere, there exists a pair of points in the grid

\[
h_{N_1,\varphi}^{(i_1-1)}\cdot\widetilde{\varphi}(h_{N_1,\varphi}^{(i_1-1)})\Z^2,
\]

that is stretched in a factor of at least
\[
(1+k_1)^{i_1}\|\Phi(1,0)-\Phi(0,0)\|>L_1.
\]

In addition, from \eqref{eq: burklosc2}, \eqref{eq: highomegacheck}, \eqref{eq: highcheck},  and Remark \ref{rmk: madic}, we have that \eqref{eq: osc2} is verified, and therefore from Proposition \ref{thm: osc} we have that $f_{N,c}^{(i_1)}$ is bi-$\omega$-realizable, but it fails to be realizable as the jacobian of any $L_1$-bi-Lipschitz homeomorphism.
\medskip

Then, if we consider a sequence of squares $(S_n)_{n\geq 1}$ in $[0,1]^2$ having disjoint interiors and converging to some point, for each $n\in\N$ we can find a function $f_{N_n,c}^{i_n}:S_n\to\{1,1+c\}$ as above that is not the jacobian of a $L_n$-bi-Lipschitz map, that admits a bi-$\omega$-regular solution for \eqref{eq: jaceq}. Let $f_{\varphi}:[0,1]^2\to\{1,1+c\}$ be defined by

\begin{equation}\label{eq: checknonbilipomeg}
    (\forall n\geq 1),\ f_{\varphi}|_{S_n}=f_{N_n,c}^{(i_n)}\qquad\text{and}\qquad f_{\varphi}(x)=1\text{ for }x\in [0,1]^2\setminus\bigcup_{n\in\N}S_n,
\end{equation}

and for each $n\in\N$ denote by $\Phi_n:S_n\to S_n$ the bi-$\omega$-mapping that is a solution of \eqref{eq: jaceq} for $f_{N_n,c}^{(i_n)}$. Finally, define $\Phi:[0,1]^2\to [0,1]^2$ by letting 
    \[
    (\forall n\geq 1),\ \Phi|_{S_n}=\Phi_n\quad\text{and}\quad\Phi(x)=x\text{ if }x\in [0,1]^2\setminus\bigcup S_n;
    \]
    we claim that $\Phi$ is the desired solution of \eqref{eq: jaceq} for $f=f_{\varphi}$.

\begin{prop}\label{prop: omegrealdens}
    Let $\varphi:(0,1)\to (0,\infty)$ be an increasing positive continuous verifying \eqref{eq: osc1} for a modulus of continuity $\omega$ asymptotically larger than Lipschitz. Then the density function $f_{\varphi}$ defined in \eqref{eq: checknonbilipomeg} admits a solution  bi-$\omega$-regular for \eqref{eq: jaceq}, but it is not realizable as the jacobian of a bi-Lipschitz homeomorphism.
\end{prop}

\begin{proof}
    It remains to prove that the sequence of mappings $(\Phi_n)_{n\geq 1}$ are bi-$\omega$-regular with uniformly bounded constants. More precisely, observe that each of the maps $\Phi_n,\Phi_n^{-1}:S_n\to S_n$ satisfy \eqref{eq: omegamap} for some positive constant $C_{\omega,n}$, then we claim that the sequence $(C_{\omega,n})_{n\geq 1}$ is uniformly bounded by above.
    \medskip
    
    From the proof of Proposition \ref{thm: osc}, and since each of the $f|_{S_n}$'s takes values in $\{1,1+c\}$, we have a sequence of constants $(a_n)_{n\geq 1}$ as in \eqref{eq: osccont1} that is uniformly bounded by below by $1$. Thus, to obtain the uniform boundedness of the $C_{\omega,n}$'s, the only parameter we need to analyze is 
    \[
    \|\nabla u_{i_0}^{(n)}\|_{L^{\infty}(C_{\overline{n},i_0})},
    \]
    
    where $u_{i_0}^{(n)}$ is defined as in \eqref{eq: ui0} with $(C_{\overline{n},i})_{\overline{n}\in\Z^d, i\in\N}$ seen as a decomposition of the square $S_n$, and where $i_0$ only depends on $\varphi$ since all the rescaling factors in \eqref{eq: highnew} to get the function $f_{N_n,c}^{(i_n)}$ that are used in the construction of $\Phi_n$ are greater or equal than $2$. If we denote these rescaling factors by $M_{n,1},\ldots,M_{n,i_n}$ (which only depend on $N_n$ and $\varphi$), then we get that
    
    \[
    \|\nabla u_{i_0}^{(n)}\|_{L^{\infty}(C_{\overline{n},i_0})}\leq\|\nabla u_0^{(n)}\|_{L^{\infty}(S_n)}\prod_{k=1}^{i_0-1}(1+C\varphi((M_{j_1}\cdot\ldots\cdot M_{j_k})^{-1})),
    \]
    
    and observe that by part iii) and iv) of Lemma \ref{lem: rivye}, each of the norms $\|\nabla u_0^{(n)}\|_{L^{\infty}(S_n)}$ is bounded by a constant depending only on $\alpha_n$, which at the same time are uniformly bounded because the definition \eqref{eq: localdens} and from the fact that $f_{N_n,c}^{(i_n)}$'s takes values in $\{1,1+c\}$. Therefore the sequence of homeomorphisms $(\Phi_n)_{n\geq 1}$ are bi-$\omega$-regular with uniformly bounded constants; this concludes the proof of Proposition \ref{prop: omegrealdens}.
    \medskip
\end{proof}

\end{appendices}

\Addresses

\begin{thebibliography}{Dillo 83}
\bibitem{linrep} {\sc J. Aliste-Prieto, D. Coronel and J.M. Gambaudo, }Linearly repetitive Delone sets are rectifiable, {\em Ann. Inst. H. Poincaré C, Anal. non Lin\'eaire}, {\bf 30}(2), 275-290, 2013.

\bibitem{aper}{\sc M. Baake and U. Grimm, }{\em Aperiodic order, vol. 1.} Cambridge University Press, 2013.

\bibitem{BurKl} {\sc D. Burago and B. Kleiner, }Separated nets in Euclidean spaces and jacobian of bi-Lipschitz maps, {\em Geom. Func. Anal.}, {\bf 8}(2), 273-282, 1998.

\bibitem{BurK2} {\sc D. Burago and B. Kleiner, }Rectifying separated nets, {\em Geom. Func. Anal.}, {\bf 12}(1), 80-92, 2002.

\bibitem{CN} {\sc M.I. Cortez and A. Navas, }Some examples of repetitive, non-rectifiable Delone sets, {\em Geom. and Top.}, {\bf 20}, 1909-1939, 2016.

\bibitem{irregsep}{\sc M. Dymond and V. Kalu\v{z}a, }Highly irregular separated nets, {\em  Israel J. Math.}, {\bf 253}(2), 501-554, 2023.

\bibitem{omegdisp}{\sc M. Dymond and V. Kalu\v{z}a, }Divergence of separated nets with respect to displacement equivalence, {\em Geom. Dedicata}, 218:1, 2023.

\bibitem{Gromov}{\sc M. Gromov, }{\em Metric structures for Riemannian and non-Riemannian spaces}. Birkh\"{a}user Boston, Vol. {\bf 152}, 1999.


\bibitem{repquasi}{\sc J. Lagarias and P. Pleasants, }Repetitive Delone sets and quasicrystals, {\em Ergodic Theory Dynam. Systems} {\bf 23}(3), 831-867, 2003.

\bibitem{McMu} {\sc C.T. McMullen, }Lipschitz maps and nets in euclidean space, {\em Geom. Func. Anal.}, {\bf 8}(2), 304-314, 1998.

\bibitem{RivYe} {\sc T. Rivi\`ere and D. Ye, }Resolutions of the prescribed volume form equation, {\em NoDEA}, {\bf 3}, 323-369, 1996.

\bibitem{quasi} {\sc D. Shechtmann, I. Bleech, D . Gratias and J.W. Cahn, }Metallic phase with long range order and no traslational symmetry, {\em Phys. Rev.Lett.}, {\bf 53(20)}, 1951-1954, 1984.

\bibitem{almostlin} {\sc Y. Smilanski and Y. Solomon, }Discrepancy and rectifiability of almost
linearly repetitive Delone sets, {\em Nonlinearity}, {\bf 35}, 6204-6217, 2022.

\bibitem{apmono} {\sc D, Smith, J.S. Myers, C.S. Kaplan and C. Goodman-Strauss, }An aperiodic monotile, {\em Combinatorial Theory}, {\bf 4(1)}, 2024.

\bibitem{RV}{\sc R. Viera, }Delone sets that are not rectifiable under Lipschitz co-uniformly continuous bijections, {\em J. Math. Anal. Appl}, {\bf 526(2)}, 2023.

\bibitem{sol}{\sc Y. Solomon, }Functions of substitution tilings as a jacobian, {\em Proc. Amer. Math, Soc.}, {\bf 141}(11), 3853-3863, 2013.
\end{thebibliography}
\end{document}